\documentclass[11pt,reqno]{amsart}
\newtheorem{theorem}{Theorem}[section]
\newtheorem{proposition}[theorem]{Proposition}

\newtheorem{remark}[theorem]{Remark}
\newtheorem{lemma}[theorem]{Lemma}

\newtheorem{definition}[theorem]{Definition}

\usepackage{amssymb}
\usepackage{graphicx}
\usepackage{amsmath}
\usepackage[all]{xy}
\usepackage{enumerate}
\usepackage{amscd}
\usepackage{cases}

\usepackage[all]{xy}
\usepackage[latin1]{inputenc} 
\usepackage{graphicx} 
\usepackage{mathtools}
\numberwithin{equation}{section}

\begin{document}

\title[Lipeng Luo\textsuperscript{1}, Yanyong Hong\textsuperscript{2} and Zhixiang Wu\textsuperscript{3}]{Extensions of finite irreducible modules of Lie conformal algebras $\mathcal{W}(a,b)$ and some Schr\"{o}dinger-Virasoro type Lie conformal algebras}

\author{Lipeng Luo\textsuperscript{1}, Yanyong Hong\textsuperscript{2} and Zhixiang Wu\textsuperscript{3}}
\address{\textsuperscript{1,3}School of Mathematical Sciences, Zhejiang University, Hangzhou, Zhejiang Province,310027,PR China.}
\address{\textsuperscript{2}Department of Mathematics, Hangzhou Normal University,
Hangzhou, 311121, P.R.China.}

\email{\textsuperscript{1}luolipeng@zju.edu.cn}
\email{\textsuperscript{2}hongyanyong2008@yahoo.com}
\email{\textsuperscript{3}wzx@zju.edu.cn}
%
\keywords{conformal algebra, conformal module, extensions}
\subjclass[2010]{17B10, 17B65, 17B68}

\date{\today}
\thanks{ This work was supported by the National Natural Science Foundation of China (No. 11871421, 11501515) and the Scientific Research Foundation of Hangzhou Normal University (No. 2019QDL012)}

\begin{abstract}
Lie conformal algebras $\mathcal{W}(a,b)$ are the semi-direct sums of Virasoro Lie conformal algebra and its nontrivial conformal modules of rank one. In this paper, we give a complete classification of extensions of finite irreducible conformal modules of $\mathcal{W}(a,b)$.  With a similar method, we characterize all extensions of finite irreducible conformal modules of Schr\"{o}dinger-Virasoro type Lie conformal algebras $TSV(a,b)$ and $TSV(c)$.
\end{abstract}

\footnote{The second author is the corresponding author.}
\maketitle

\section{Introduction}\label{intro}
The notion of a Lie conformal algebra, which was introduced by Kac in \cite{KacV, KacF}, represents an axiomatic description of the operator product expansion (or rather its Fourier transform) of chiral fields in conformal field theory (see \cite{BPZ}). It has been shown that the theory of Lie conformal algebras has close connections to vertex algebras, infinite-dimensional Lie algebras satisfying the locality property in \cite{KacL} and Hamiltonian formalism in the theory of nonlinear evolution equations (see \cite{BDK}). It is known that Virasoro Lie conformal algebra $Vir$ and current Lie conformal algebra $Cur\mathcal{G}$ associated to a finite-dimensional simple Lie algebra $\mathcal{G}$ exhaust all finite simple Lie conformal algebras (see \cite{DK}). Moreover, all finite irreducible conformal modules of finite simple Lie conformal algebras were characterized in \cite{CK}. In general, conformal modules of Lie conformal algebras including finite simple Lie conformal algebras are not completely reducible.
Therefore, it is necessary to investigate the extension problem of finite irreducible conformal modules of Lie conformal algebras. Extensions between finite irreducible conformal modules over the Virasoro, the current, the Neveu-Schwarz and the semi-direct sum of the Virasoro and the current conformal algebras were classified by Cheng, Kac and Wakimoto in \cite{CKW1,CKW2}. Ngau Lam in \cite{LN} classified extensions between finite irreducible conformal modules over the supercurrent conformal algebras by using the techniques developed in \cite{CKW1}.

In this paper, we  investigate extensions of finite irreducible conformal modules of Lie conformal algebras $\mathcal{W}(a,b)$, $TSV(a,b)$ and $TSV(c)$, where $\mathcal{W}(a,b)$ is a semi-direct sum of $Vir$ and its nontrivial conformal modules of rank one, $TSV(a,b)$ and $TSV(c)$ are two classes of Schr\"{o}dinger-Virasoro type Lie conformal algebras introduced in \cite{Hong}.  Note that $\mathcal{W}(1-b,0)$ is just the Lie conformal algebra $\mathcal{W}(b)$ in \cite{Xu-Yue},  $\mathcal{W}(1,0)$ is just the Heisenberg-Virasoro Lie conformal algebra,  $TSV(\frac{3}{2},0)$ is just the Schr\"{o}dinger-Virasoro Lie conformal algebra in \cite{Su-Yuan}
and $TSV(0,0)$ is just the Schr\"{o}dinger-Virasoro type Lie conformal algebra in \cite{Wang-Xu-Xia}. Finite irreducible conformal modules of $\mathcal{W}(1,0)$ and $\mathcal{W}(1-b,0)$ were classified in \cite{Wu-Yuan}.  In \cite{Luo-Hong-Wu}, we  gave  a complete classification of finite irreducible conformal modules of  $\mathcal{W}(a,b)$, $TSV(a,b)$ and $TSV(c)$. In \cite{Ling-Yuan1, Ling-Yuan2, Yuan-Ling}, Ling and Yuan classified all extensions of finite irreducible conformal modules over $\mathcal{W}(1,0)$ , $\mathcal{W}(1-b,0)$ and $TSV(\frac{3}{2},0)$. In this paper, we deal with the same problem for $\mathcal{W}(a,b)$, $TSV(a,b)$ and $TSV(c)$. According to the definitions of  $TSV(a,b)$ and $TSV(c)$ (see Definition \ref{deff1}), it is easy to see that $\mathcal{W}(a,b)$ is isomorphic to $TSV(a,b)/\mathbb{C}[\partial]M$ and $\mathcal{W}(\frac{3}{2},c)$ is isomorphic to $TSV(c)/\mathbb{C}[\partial]M$. So, the extensions of finite irreducible conformal modules of $TSV(a,b)$ and $TSV(c)$ are closely related with those of $\mathcal{W}(a,b)$. Therefore, we first investigate all extensions of finite irreducible conformal modules of $\mathcal{W}(a,b)$ and then give a complete classification of all extensions of finite irreducible conformal modules of $TSV(a,b)$ and $TSV(c)$.

The rest of this paper is organized as follows. In Section 2, we introduce some basic definitions, notations, and related known results about Virasoro Lie conformal algebra $Vir$. In Section 3, we first recall all finite nontrivial irreducible conformal modules of $\mathcal{W}(a,b)$. Then we give a complete classification of all extensions of finite irreducible conformal modules of $\mathcal{W}(a,b)$. In Section 4, we recall all finite nontrivial irreducible conformal modules over Lie conformal algebras $TSV(a,b)$ and $TSV(c)$. Moreover, we also classify all extensions of finite irreducible conformal modules over them by using the results and methods given in Section 3.

Throughout this paper, we use $\mathbb{C}$ to represent the set of complex numbers. In addition, all vector spaces and tensor products are over $\mathbb{C}$.

\section{preliminaries}

 In this section, we recall some basic definitions and related results about Lie conformal algebras and fix some notations for later use. For a detailed description, one can refer to \cite{CK,CKW1,KacV,Luo-Hong-Wu}.
\begin{definition}
\begin{em}
A \emph {Lie conformal algebra} $\mathcal{R}$ is a $\mathbb{C}[\partial]$-module endowed with a $\mathbb{C}$-linear map from $\mathcal{R}\otimes\mathcal{R}$ to $\mathcal{R}[\lambda], a\otimes b \mapsto [a_\lambda b]$, called the $\lambda$-bracket, satisfying the following axioms:
\begin{align}
    [\partial a_\lambda b]&=-\lambda[a_\lambda b],\quad  [ a_\lambda \partial b]=(\partial+\lambda)[a_\lambda b] \quad (conformal \  sesquilinearity),\\
    {}[a_\lambda b]&=-[b_{-\lambda-\partial} a]\quad (skew\text{-}symmetry),\\
    {}[a_\lambda [b_\mu c]]&=[[a_\lambda b]_{\lambda+\mu} c]+[ b_\mu [a_\lambda c]]\quad (Jacobi\  identity),
\end{align}
for $a,b,c \in\mathcal{R}$.
\end{em}		
\end{definition}

A Lie conformal algebra $\mathcal{R}$ is called \emph {finite} if $\mathcal{R}$ is finitely generated as a $\mathbb{C}[\partial]$-module. The \emph {rank} of a Lie conformal algebra $\mathcal{R}$, denoted by rank($\mathcal{R}$), is its rank as a $\mathbb{C}[\partial]$-module.

\begin{definition}\label{Def22}
\begin{em}
A \emph {conformal module} $M$ over a Lie conformal algebra $\mathcal{R}$ is a $\mathbb{C}[\partial]$-module endowed with a $\mathbb{C}$-linear map $\mathcal{R}\otimes M \rightarrow M[\lambda], a\otimes v \mapsto a_\lambda v$, satisfying the following conditions:
\begin{align}
    &(\partial a)_\lambda v=-\lambda a_\lambda v,\quad  a_\lambda (\partial v)=(\partial+\lambda)a_\lambda v, \quad \\
    &{}a_\lambda (b_\mu v)- b_\mu (a_\lambda v)=[a_\lambda b]_{\lambda+\mu} v,\quad
\end{align}
for $a,b \in\mathcal{R}, v\in M$.
\end{em}	

Suppose $M$, $N$ are two $\mathcal{R}$-modules. Then a $\mathbb{C}[\partial]$-module homomorphism $\varphi$ from $M$ to $N$  is said to be a homomorphism of $\mathcal{R}$-modules if $\varphi(a_{\lambda}m)=a_{\lambda}\varphi(m)$ for all $m\in M$ and $a\in \mathcal{R}$.

\end{definition}

A conformal module $M$ over a Lie conformal algebra $\mathcal{R}$  is also called a representation of $\mathcal{R}$, or an $\mathcal{R}$-module. If $M$ is finitely generated over $\mathbb{C}[\partial]$, then
it  is simply called \emph {finite}. Furthermore, if $M$ is free over $\mathbb{C}[\partial]$ and finite, then  the \emph {rank} of $M$ is its rank as a $\mathbb{C}[\partial]$-module. A conformal module $M$ is said to be  \emph {irreducible} if it has no nonzero submodules $N$ such that $N\neq M$.

Let $\mathcal{R}$ be a Lie conformal algebra and $M$ an $\mathcal{R}$-module. An element $m\in M$ is called \emph {invariant} if $\mathcal{R}_\lambda m=0$. Obviously, the set of all invariants of $M$ is a conformal submodule of $M$, denoted by $M^0$. An $\mathcal{R}$-module $M$ is called trivial if $M^0=M$, i.e., a module on which $\mathcal{R}$ acts trivially. For any $\eta \in \mathbb{C}$, we obtain a natural trivial  $\mathcal{R}$-module $\mathbb{C}c_{\eta}$, which is determined by $\eta$, such that $\mathbb{C}c_{\eta}=\mathbb{C}$ and $\partial c_{\eta}=\eta c_{\eta}, \mathcal{R}_\lambda c_{\eta}=0$. It is easy to check that the modules $\mathbb{C}c_{\eta}$ with $\eta \in \mathbb{C}$ exhaust all trivial irreducible $\mathcal{R}$-modules.

\begin{definition}
\begin{em}
Let $V$ and $W$ be two modules over a Lie conformal algebra (or a Lie algebra) $\mathcal{R}$. An \emph {extension} of $W$ by $V$ is an exact sequence of $\mathcal{R}$-modules of the form
\begin{align}
0\longrightarrow V \stackrel{i}{\longrightarrow} E \stackrel{p}{\longrightarrow}W \longrightarrow 0,
\end{align}
where $E$ is isomorphic to $V\oplus W$ as a vector space.
Two extensions $0\longrightarrow V \stackrel{i}{\longrightarrow} E \stackrel{p}{\longrightarrow}W \longrightarrow 0$ and $0\longrightarrow V \stackrel{i'}{\longrightarrow} E' \stackrel{p'}{\longrightarrow}W \longrightarrow 0$ are said to be equivalent if there exists a homomorphism of modules such that the following diagram commutes
\begin{align}
  \begin{CD}
  0 @>>> V   @>i>>   E   @>p>>   W   @>>>0  \\
     @.      @V{1_V}VV       @V{\Psi}VV       @V{1_W}VV  @.    \\
  0 @>>> V   @>i'>>   E'   @>p'>>   W  @>>>0.
  \end{CD}
\end{align}

\end{em}		
\end{definition}

Obviously, the direct sum of modules $V\oplus W$ gives rise to an extension $0\to V\to V\oplus W\to W\to 0$.  Any extension $0\to V\to E\to W\to 0 $, which is equivalent to  $0\rightarrow V\to V\oplus W\to W\to 0$, is  called \emph{trivial extensions}.

In general, taking Lie algebra as an example, an extension can be thought of as the direct sum of vector spaces $E=V\oplus W$, where $V$ is a submodule of $E$, while for $w\in W$ we have:
\begin{align*}
a\cdot w=aw+f_a(w),\   a\in \mathcal{R},
\end{align*}
where $f_a: W\to V$ is a linear map satisfying the cocycle condition:
\begin{align*}
f_{[a,b]}(w)=f_a(bw)+af_b(w)-f_b(aw)-bf_a(w),\   b\in \mathcal{R}.
\end{align*}
The set of these cocycles forms a vector space $\mathcal{E}xt(W,V)$ over $\mathbb{C}$. Cocycles equivalent to trivial extension are called \emph{coboundaries}. They form a subspace $\mathcal{E}xt^c(W,V)$  and the quotient space $\mathcal{E}xt(W,V)/\mathcal{E}xt^c(W,V)$ is denoted by $Ext(W,V)$.

It was shown in \cite{CK} that
 \begin{proposition}\label{pr1} Let $Vir=\mathbb{C}[\partial]L$ be the Virasoro Lie conformal algebra. Then
 all free nontrivial $Vir$-modules of rank one over $\mathbb{C}[\partial]$ are as follows($ \alpha,\beta \in \mathbb{C}$):
  \begin{equation}
  \xymatrix{M_{\alpha,\beta}=\mathbb{C}[\partial]v,\qquad L_\lambda v=(\partial+\alpha\lambda+\beta)v.}
 \end{equation}
 Moreover, the module $M_{\alpha,\beta}$ is irreducible if and only if $\alpha$ is non-zero. The module $M_{0,\beta}$ contains a unique nontrivial submodule $ (\partial+\beta)M_{0,\beta}$ isomorphic to $M_{1,\beta}$. The modules $M_{\alpha,\beta}$ with $\alpha\neq 0$ exhaust all finite irreducible nontrivial $Vir$-modules.
 \end{proposition}
Therefore, $M_{\alpha,\beta}$ with $\alpha\neq 0$, together with the one-dimensional modules $\mathbb{C}c_{\eta}({\eta} \in \mathbb{C})$, form a complete list of finite irreducible conformal modules over the Virasoro conformal algebra.

In \cite{CKW1}, extensions over the Virasoro conformal modules of the following types have been classified:
\begin{align}
&0\longrightarrow \mathbb{C}c_{\eta} \longrightarrow  E\longrightarrow  M_{\alpha,\beta} \longrightarrow   0\label{V1}\\
&0\longrightarrow M_{\alpha,\beta}\longrightarrow  E\longrightarrow  \mathbb{C}c_{\eta} \longrightarrow   0\label{V2}\\
&0\longrightarrow M_{\bar{\alpha},\bar{\beta}}\longrightarrow  E\longrightarrow  M_{\alpha,\beta}\longrightarrow   0\label{V3}.
\end{align}

We list the corresponding results in the following three theorems for later use.
\begin{theorem}\label{t25}(Ref. \cite{CKW1}, Proposition 2.1)
Nontrivial extensions of Virasoro conformal modules of the form (\ref{V1}) exist if and only if $\beta+\eta=0$ and $\alpha=1$ or $2$. In these cases, they are given (up to equivalence) by
\begin{eqnarray*}
L_\lambda v_{\alpha}=(\partial+\alpha\lambda+\beta)v_{\alpha}+f(\lambda)c_{\eta},
\end{eqnarray*}
where
\begin{enumerate}[(i)]
\item $f(\lambda)=c_2\lambda^2$, for $\alpha=1$ and $c_2\neq 0$.
\item $f(\lambda)=c_3\lambda^3$, for $\alpha=2$ and $c_3\neq 0$.
\end{enumerate}
Furthermore, all trivial cocycles are given by scalar multiples of the polynomial $f(\lambda)=\alpha \lambda + \beta+\eta$.
\end{theorem}

\begin{theorem}\label{t2}(Ref.  \cite{CKW1}, Proposition 2.2)
Nontrivial extensions of Virasoro conformal modules of the form (\ref{V2}) exist if and only if $\beta+\eta=0$ and $\alpha=1$. In these cases, they are given (up to equivalence) by
\begin{eqnarray*}
L_\lambda c_{\eta}=f(\partial,\lambda)v_{\alpha}, \quad \partial c_{\eta}=\eta c_{\eta}+p(\partial)v_{\alpha},
\end{eqnarray*}
where $f(\partial,\lambda)=p(\partial)=k$ for some nonzero  $k\in\mathbb{C}$.

Furthermore, all trivial cocycles are given by the same scalar multiples of the polynomial $f(\partial,\lambda)=(\partial+\alpha \lambda + \beta)\phi(\partial+\lambda)$ and $p(\partial)=(\partial-\eta)\phi(\partial)$, where $\phi$ is a polynomial.
\end{theorem}

\begin{theorem}\label{t3}(Ref. \cite{CKW1}, Theorem 3.1)
Nontrivial extensions of Virasoro conformal modules of the form (\ref{V3}) exist if and only if $\beta=\bar{\beta}$ and $\alpha-\bar{\alpha}=0,1,2,3,4,5,6$. In these cases, they are given (up to equivalence) by
\begin{eqnarray*}
L_\lambda v_{\alpha}=(\partial+\alpha\lambda+\beta)v_{\alpha}+f(\partial,\lambda)v_{\bar{\alpha}}.
\end{eqnarray*}
The complete list of values of $\alpha$ and $\bar{\alpha}$ along with the corresponding polynomials $f(\partial,\lambda)$, is given as follows, whose nonzero scalar multiples give rise to nontrivial extensions (by replacing $\partial$ by $\partial+\beta$):
\begin{enumerate}[(i)]
\item $\alpha=\bar{\alpha}$ with $\alpha \in \mathbb{C}$. $f(\partial,\lambda)=a_0+a_1\lambda$, where $(a_0,a_1)\neq (0,0)$.
\item $\alpha=1$ and $\bar{\alpha}=0$. $f(\partial,\lambda)=a_0\partial+b_0\partial\lambda+b_1\lambda^2$, where $(a_0,b_0,b_1)\neq (0,0,0)$.
\item $\alpha-\bar{\alpha}=2$ with $\alpha \in \mathbb{C}$. $f(\partial,\lambda)=\lambda^2(2\partial+\lambda)$.
\item $\alpha-\bar{\alpha}=3$ with $\alpha \in \mathbb{C}$. $f(\partial,\lambda)=\partial\lambda^2(\partial+\lambda)$.
\item $\alpha-\bar{\alpha}=4$ with $\alpha \in \mathbb{C}$. $f(\partial,\lambda)=\lambda^2(4\partial^3+6\partial^2\lambda-\partial\lambda^2+\bar{\alpha}\lambda^3)$.

\item $\alpha=5$ and $\bar{\alpha}=0$. $f(\partial,\lambda)=5\partial^4\lambda^2+10\partial^2\lambda^4-\partial\lambda^5$.
\item $\alpha=1$ and $\bar{\alpha}=-4$. $f(\partial,\lambda)=\partial^4\lambda^2-10\partial^2\lambda^4-17\partial\lambda^5-8\lambda^6$.

\item $\alpha=\frac{7}{2}\pm\frac{\sqrt{19}}{2}$ and $\bar{\alpha}=-\frac{5}{2}\pm\frac{\sqrt{19}}{2}$. $f(\partial,\lambda)=\partial^4\lambda^3-(2\bar{\alpha}+3)\partial^3\lambda^4-3\bar{\alpha}\partial^2\lambda^5-(3\bar{\alpha}+1)\partial\lambda^6-(\bar{\alpha}+\frac{9}{28})\lambda^7$.

\end{enumerate}
Furthermore, all trivial cocycles are given by scalar multiples of the polynomial $f(\partial,\lambda)=(\partial+\alpha \lambda + \beta)\phi(\partial)-(\partial+\bar{\alpha} \lambda + \bar{\beta})\phi(\partial+\lambda)$, where $\phi$ is a polynomial.
\end{theorem}

\begin{remark}
We keep the part of $\bar{\alpha}=0$ in Theorem \ref{t3} for later use.
\end{remark}
\section{Extensions of finite irreducible $\mathcal{W}(a,b)$-modules}

In this section, we introduce the definition of Lie conformal algebra $\mathcal{W}(a,b)$ and give a complete classification of extensions of finite irreducible $\mathcal{W}(a,b)$-modules.

\begin{definition}
The Lie conformal algebra $\mathcal{W}(a,b)$ with two parameters $a$, $b\in \mathbb{C}$ is a free $\mathbb{C}[\partial]$-module generated by $L$ and $W$ satisfying
\begin{eqnarray*}
[L_\lambda L]=(\partial+2\lambda)L,~~~~[L_\lambda W]=(\partial+a\lambda+b)W,~~~~[W_\lambda W]=0.
\end{eqnarray*}
\end{definition}
All finite nontrivial conformal modules over the Lie conformal algebra $\mathcal{W}(a,b)$ were classified in \cite{Luo-Hong-Wu}. We recall them  via  the following theorem.

\begin{theorem}\label{t31}(Ref. \cite{Luo-Hong-Wu}, Theorem 3.10)
Any finite nontrivial irreducible $\mathcal{W}(a,b)$-module $M$ is free of rank one over $\mathbb{C}[\partial]$. Moreover,
\begin{enumerate}
\item If $(a,b)\neq (1,0)$,
  \begin{equation*}
  \xymatrix{M\cong M_{\alpha,\beta}=\mathbb{C}[\partial]v,\  L_\lambda v=(\partial+\alpha\lambda+\beta)v,\  W_\lambda v=0,}
   \end{equation*}
with $\alpha,\beta \in \mathbb{C}$  and  $\alpha \neq 0$.

\item If $(a,b)=(1,0)$,
 \begin{equation*}
  \xymatrix{M\cong M_{\alpha,\beta,\gamma}=\mathbb{C}[\partial]v,\  L_\lambda v=(\partial+\alpha\lambda+\beta)v,\  W_\lambda v=\gamma v,}
 \end{equation*}
with $\alpha,\beta,\gamma \in \mathbb{C}$ and $(\alpha,\gamma) \neq (0,0)$.
\end{enumerate}
\end{theorem}

In this paper, we denote the $\mathcal{W}(a,b)$-module $M$ from Theorem \ref{t31} by $M_{\alpha,\beta}$ if $(a,b)\neq (1,0)$, and $M_{\alpha,\beta,\gamma}$ if $(a,b)=(1,0)$, respectively. Actually, $\mathcal{W}(1,0)$ is the Heisenberg-Virasoro conformal algebra. Moreover, extensions of finite irreducible modules over it were classified in \cite{Ling-Yuan1,Yuan-Ling}. So we will give their results directly below without proof.

By Definition \ref{Def22}, a $\mathcal{W}(a,b)$-module structure on $M$ is given by $L_{\lambda}, W_{\lambda} \in End_{\mathbb{C}}(M)[\lambda]$ such that
\begin{align}
&[L_{\lambda},L_{\mu}]=(\lambda-\mu)L_{\lambda+\mu}, \label{f31}\\
&[L_{\lambda},W_{\mu}]=((a-1)\lambda-\mu+b)W_{\lambda+\mu}, \label{f32}\\
&[W_{\lambda},L_{\mu}]=-((a-1)\mu-\lambda+b)W_{\lambda+\mu}, \label{f33}\\
&[W_{\lambda},W_{\mu}]=0, \label{f34}\\
&[\partial,L_{\lambda}]=-\lambda L_{\lambda}, \label{f35}\\
&[\partial,W_{\lambda}]=-\lambda W_{\lambda}. \label{f36}
\end{align}

First, we consider extensions of finite irreducible $\mathcal{W}(a,b)$-modules of the form
\begin{align}
0\longrightarrow \mathbb{C}c_{\eta} \longrightarrow  E\longrightarrow  M \longrightarrow   0\label{W1}
\end{align}
Since $M$ is free as a $\mathbb{C}[\partial]$-module, $E$ as a $\mathbb{C}[\partial]$-module in (\ref{W1}) is isomorphic to $\mathbb{C}c_{\eta}\oplus M$, where $\mathbb{C}c_{\eta}$ is a $\mathcal{W}(a,b)$-submodule, and $M=\mathbb{C}[\partial]v_{\alpha}$ such that the following identities hold in $E$:
\begin{enumerate}
\item If $(a,b)\neq (1,0)$,
  \begin{equation}
  L_\lambda v_{\alpha}=(\partial+\alpha\lambda+\beta)v_{\alpha}+f(\lambda)c_{\eta},\quad W_\lambda v_{\alpha}=g(\lambda)c_{\eta};\label{W11}
   \end{equation}

  \item If $(a,b)=(1,0)$,
 \begin{equation}
 L_\lambda v_{\alpha}=(\partial+\alpha\lambda+\beta)v_{\alpha}+f(\lambda)c_{\eta},\quad W_\lambda v_{\alpha}=\gamma v_{\alpha}+g(\lambda)c_{\eta};\label{W12}
  \end{equation}
\end{enumerate}
where $f(\lambda), g(\lambda) \in \mathbb{C}[\lambda]$.

\begin{lemma}\label{l32}
All trivial extensions of finite irreducible $\mathcal{W}(a,b)$-modules of the form (\ref{W1}) are given by (\ref{W11}) and (\ref{W12}), and
\begin{enumerate}
\item If $(a,b)\neq (1,0)$, $f(\lambda)$ is a scalar multiple of $\alpha\lambda+\beta+\eta$ and $g(\lambda)=0$.
\item If $(a,b)=(1,0)$, $f(\lambda)$ and $g(\lambda)$ are the same scalar multiple of $\alpha\lambda+\beta+\eta$ and $\gamma$, respectively.
\end{enumerate}
\end{lemma}
\begin{proof}
(1) Assume that (\ref{W1}) is a trivial extension, i.e., there exists $v_{\alpha}'=\varphi(\partial)v_{\alpha}+kc_{\eta}\in E$, where $k \in \mathbb{C}$ and $0\neq\varphi(\partial) \in \mathbb{C}[\partial]$, such that
\begin{eqnarray*}
L_\lambda v_{\alpha}'=(\partial+\alpha\lambda+\beta)v_{\alpha}'=(\partial+\alpha\lambda+\beta)\varphi(\partial)v_{\alpha}+k(\alpha\lambda+\beta+\eta)c_{\eta},\quad W_\lambda v_{\alpha}'=0.
\end{eqnarray*}
On the other hand, it follows from (\ref{W11}) that
\begin{eqnarray*}
&&L_\lambda v_{\alpha}'=(\partial+\alpha\lambda+\beta)\varphi(\partial+\lambda)v_{\alpha}+f(\lambda)\varphi(\partial+\lambda)c_{\eta},\\
&&W_\lambda v_{\alpha}'=g(\lambda)\varphi(\partial+\lambda)c_{\eta}.
\end{eqnarray*}
We can obtain that $\varphi(\partial)$ is a nonzero constant and $g(\lambda)=0$ by comparing both expressions for $L_\lambda v_{\alpha}'$ and $W_\lambda v_{\alpha}'$, respectively. Thus $f(\lambda)$ is a scalar multiple of $\alpha\lambda+\beta+\eta$.

(2) See Corollary 6.1 in \cite{Yuan-Ling}.
\end{proof}

\begin{theorem}\label{t33}
(1) If $(a,b)\neq (1,0)$, nontrivial extensions of finite irreducible $\mathcal{W}(a,b)$-modules of the form (\ref{W1}) exist. Moreover, they are given (up to equivalence) by (\ref{W11}). The values of $\beta$ and $\eta$ along with the pairs of polynomials $g(\lambda)$ and $f(\lambda)$, whose nonzero scalar multiples give rise to nontrivial extensions, are listed as follows:

\begin{enumerate}[(i)]
\item if $g(\lambda)=0$, then $\alpha=1,2$, $\beta+\eta=0$ and $f(\lambda)$ is from the nonzero polynomials of Theorem \ref{t25};
\item if $a\neq1$, $b=0$ and $\beta+\eta =0$, then $g(\lambda)=k$ for some nonzero complex number $k$, $\alpha=1-a$, and

$$f(\lambda)=
\begin{cases}
c_2{\lambda}^2,  &\alpha=1,
\cr c_3\lambda^3,  &\alpha=2,
\cr 0,  &otherwise,
\end{cases}$$
with $c_2, c_3 \in \mathbb{C}$;

\item if $a\neq1$, $b+\beta+\eta=0$ and $\beta+\eta \neq0$, then $g(\lambda)=k$ for some nonzero complex number $k$, $\alpha=1-a$, and $f(\lambda)=0$;
\item if $a=1, b\neq0$ and $b+\beta+\eta=0$, then $g(\lambda)=k(1-\frac{1}{b}\lambda)$ for some nonzero complex number $k$, $\alpha=1$, and $f(\lambda)=0$.
\end{enumerate}
(2) If $(a,b)=(1,0)$, nontrivial extensions of finite irreducible $\mathcal{W}(1,0)$-modules of the form (\ref{W1}) exist if and only if $\beta+\eta=0$ and $\gamma=0$. Moreover, they are given (up to equivalence) by (\ref{W12}), where, if $g(\lambda)=0$, then $\alpha=1,2$ and $f(\lambda)$ is from the nonzero polynomials of Theorem \ref{t25}, or else $g(\lambda)=k\lambda$ for some nonzero complex number $k$, $\alpha=1$ and $f(\lambda)=c_2{\lambda}^2$ with $c_2 \in \mathbb{C}$.

\end{theorem}
\begin{proof}
(1) Applying both sides of (\ref{f31}) and (\ref{f32}) to $v_{\alpha}$, we obtain
\begin{eqnarray}
&&(\lambda-\mu)f(\lambda+\mu)=(\alpha\mu+\lambda+\beta+\eta)f(\lambda)-(\alpha\lambda+\mu+\beta+\eta)f(\mu), \label{f310}\\
&&((a-1)\lambda-\mu+b)g(\lambda+\mu)=-(\alpha\lambda+\mu+\beta+\eta)g(\mu). \label{f311}
\end{eqnarray}
Setting $\lambda=0$ in (\ref{f311}) gives
\begin{eqnarray}
(b+\beta+\eta)g(\mu)=0. \label{f312}
\end{eqnarray}

\textbf{Case 1.} $b+\beta+\eta \neq0$.

By (\ref{f312}), $g(\mu)=0$. It reduces to the case of Virasoro conformal algebra. We can obtain the result by Theorem \ref{t25}.

\textbf{Case 2.} $b+\beta+\eta =0$.

If $g(\lambda)=0$, then it reduces to the case of Virasoro conformal algebra. We obtain the result by Theorem \ref{t25}.

Now we assume that $g(\lambda)\neq0$. Setting $\mu=0$ in (\ref{f311}), we obtain that
\begin{eqnarray}
((a-1)\lambda+b)g(\lambda)=-(\alpha\lambda+\beta+\eta)g(0). \label{eq313}
\end{eqnarray}

By solving the equation (\ref{eq313}) and combining Theorem \ref{t25} and Lemma \ref{l32}, we can draw the following conclusions:

 If $a\neq1$ and $\beta+\eta =0$, then $g(\lambda)=k$ for some nonzero complex number $k$, $\alpha=1-a$, and

$$f(\lambda)=
\begin{cases}
c_2{\lambda}^2,  &\alpha=1,
\cr c_3\lambda^3,  &\alpha=2,
\cr 0,  &otherwise,
\end{cases}$$
with $c_2, c_3 \in \mathbb{C}$.

If $a\neq1$ and $\beta+\eta \neq0$, then $g(\lambda)=k$ for some nonzero complex number $k$, $\alpha=1-a$, and $f(\lambda)=0$.

If $a=1$ and $b\neq 0$, then $\beta+\eta \neq0$, then $g(\lambda)=k(1-\frac{1}{b}\lambda)$, $\alpha=1$ for some nonzero complex number $k$ and $f(\lambda)=0$.

(2) See Corollary 6.1 in \cite{Yuan-Ling}.

This completes the proof.
\end{proof}
Next, we consider extensions of finite irreducible $\mathcal{W}(a,b)$-modules of the form
\begin{align}
0\longrightarrow M \longrightarrow  E\longrightarrow  \mathbb{C}c_{\eta} \longrightarrow   0\label{W2}.
\end{align}
As we described in the Section 2, $E$ as a vector space in (\ref{W2}) is isomorphic to $ M\oplus \mathbb{C}c_{\eta}$, where $M$ is a $\mathcal{W}(a,b)$-submodule, and $M=\mathbb{C}[\partial]v_{\alpha}$ such that the following identities hold in $E$:
\begin{equation}
  L_\lambda c_{\eta}=f(\partial,\lambda)v_{\alpha},\quad W_\lambda c_{\eta}=g(\partial,\lambda)v_{\alpha},\quad \partial c_{\eta}=\eta c_{\eta} +p(\partial)v_{\alpha}, \label{W21}
 \end{equation}
 where $f(\partial,\lambda), g(\partial,\lambda) \in \mathbb{C}[\partial,\lambda]$ and $p(\partial) \in \mathbb{C}[\partial]$.

\begin{lemma}\label{l34}
All trivial extensions of finite irreducible $\mathcal{W}(a,b)$-modules of the form (\ref{W2}) are given by (\ref{W21}), and
\begin{enumerate}
\item If $(a,b)\neq (1,0)$, $f(\partial,\lambda)=(\partial+\alpha\lambda+\beta)\phi(\partial+\lambda)$, $g(\partial,\lambda)=0$ and $p(\partial)=(\partial-\eta)\phi(\partial)$, where $\phi$ is a polynomial.
\item If $(a,b)=(1,0)$, $f(\partial,\lambda)=(\partial+\alpha\lambda+\beta)\phi(\partial+\lambda)$, $g(\partial,\lambda)=\gamma \phi(\partial+\lambda)$ and $p(\partial)=(\partial-\eta)\phi(\partial)$, where $\phi$ is a polynomial.
\end{enumerate}
\end{lemma}

\begin{proof}
(1) Assume that (\ref{W2}) is a trivial extension, i.e., there exists $c_{\eta}'=kc_{\eta}+\phi(\partial)v_{\alpha}\in E$, where $0\neq k \in \mathbb{C}$ and $\phi(\partial) \in \mathbb{C}[\partial]$, such that $L_\lambda c_{\eta}'=W_\lambda c_{\eta}'=0$ and $\partial c_{\eta}'=\eta c_{\eta}'$. On the other hand, it follows from (\ref{W21}) that
\begin{eqnarray*}
&&L_\lambda c_{\eta}'=(\partial+\alpha\lambda+\beta)\phi(\partial+\lambda)v_{\alpha}+kf(\partial,\lambda)v_{\alpha},\\
&&W_\lambda c_{\eta}'=kg(\partial,\lambda)v_{\alpha},\\
&&\partial c_{\eta}'=k\eta c_{\eta}+(kp(\partial)+\partial\phi(\partial))v_{\alpha}.
\end{eqnarray*}
We obtain the result by comparing both expressions for $L_\lambda c_{\eta}'$, $W_\lambda c_{\eta}'$ and $\partial c_{\eta}'$, respectively.

(2) See Corollary 6.2 in \cite{Yuan-Ling}.
\end{proof}

\begin{theorem}\label{t35}
(1) If $(a,b)\neq (1,0)$, nontrivial extensions of finite irreducible $\mathcal{W}(a,b)$-modules of the form (\ref{W2}) exist if and only if $\beta+\eta=0$ and $\alpha=1$. In this case, $dim Ext(\mathbb{C}c_{-\beta}, M_{1,\beta})=1$, and the unique (up to equivalence) nontrivial extension is given by
\begin{align*}
L_\lambda c_{\eta}=kv_{\alpha},\quad W_\lambda c_{\eta}=0,\quad \partial c_{\eta}=\eta c_{\eta}+kv_{\alpha},
\end{align*}
where $k$ is a nonzero complex number.\\
(2) If $(a,b)=(1,0)$, nontrivial extensions of finite irreducible $\mathcal{W}(1,0)$-modules of the form (\ref{W2}) exist if and only if $\beta+\eta=0$ and $(\alpha,\gamma)=(1,0)$. In this case, $dim Ext(\mathbb{C}c_{-\beta}, M_{1,\beta,0})=1$, and the unique (up to equivalence) nontrivial extension is given by
\begin{align*}
L_\lambda c_{\eta}=kv_{\alpha},\quad W_\lambda c_{\eta}=0,\quad \partial c_{\eta}=\eta c_{\eta}+kv_{\alpha},
\end{align*}
where $k$ is a nonzero complex number.
\end{theorem}
\begin{proof}
(1) Applying both sides of (\ref{f31}), (\ref{f35}) and (\ref{f36}) to $c_{\eta}$ gives the following equations:
\begin{eqnarray}
&&(\partial+\alpha\lambda+\beta)f(\partial+\lambda,\mu)-(\partial+\alpha\mu+\beta)f(\partial+\mu,\lambda)=(\lambda-\mu)f(\partial,\lambda+\mu),\label{f319}\\
&&(\partial+\lambda-\eta)f(\partial,\lambda)=(\partial+\alpha\lambda+\beta)p(\partial+\lambda),\label{f320}\\
&&(\partial+\lambda-\eta)g(\partial,\lambda)=0. \label{f321}
\end{eqnarray}
Obviously, $g(\partial,\lambda)=0$ by (\ref{f321}). This reduces to the case of Virasoro conformal algebra. We can obtain the result by Theorem \ref{t2}.

(2) See Corollary 6.2 in \cite{Yuan-Ling}.

This completes the proof.
\end{proof}

Finally, we consider extensions of finite irreducible $\mathcal{W}(a,b)$-modules of the form
\begin{align}
0\longrightarrow \bar{M} \longrightarrow E\longrightarrow  M \longrightarrow   0\label{W3}
\end{align}
Since $M$ is free as a $\mathbb{C}[\partial]$-module, $E$ as a $\mathbb{C}[\partial]$-module in (\ref{W3}) is isomorphic to $ \bar{M}\oplus M$, where $\bar{M}$ is a $\mathcal{W}(a,b)$-submodule, and $\bar{M}=\mathbb{C}[\partial]v_{\bar{\alpha}}$,  $M=\mathbb{C}[\partial]v_{\alpha}$ such that the following identities hold in $E$:

\begin{enumerate}
\item If $(a,b)\neq (1,0)$,
  \begin{equation}
  L_\lambda v_{\alpha}=(\partial+\alpha\lambda+\beta)v_{\alpha}+f(\partial,\lambda)v_{\bar{\alpha}},\quad W_\lambda v_{\alpha}=g(\partial,\lambda)v_{\bar{\alpha}};\label{W31}
   \end{equation}

  \item If $(a,b)=(1,0)$,
 \begin{equation}
 L_\lambda v_{\alpha}=(\partial+\alpha\lambda+\beta)v_{\alpha}+f(\partial,\lambda)v_{\bar{\alpha}},\quad W_\lambda v_{\alpha}=\gamma v_{\alpha}+g(\partial,\lambda)v_{\bar{\alpha}};\label{W32}
  \end{equation}
\end{enumerate}
where $f(\partial,\lambda), g(\partial,\lambda) \in \mathbb{C}[\partial,\lambda]$.

\begin{lemma}\label{l36}
All trivial extensions of finite irreducible $\mathcal{W}(a,b)$-modules of the form (\ref{W3}) are given by (\ref{W31}) and (\ref{W32}), and
\begin{enumerate}
\item If $(a,b)\neq (1,0)$, $f(\partial,\lambda)$ is a scalar multiple of $(\partial+\alpha\lambda+\beta)\phi(\partial)-(\partial+\bar{\alpha}\lambda+\bar{\beta})\phi(\partial+\lambda)$ and $g(\partial,\lambda)=0$, where $\phi$ is a polynomial.
\item If $(a,b)=(1,0)$, $f(\partial,\lambda)$ and $g(\partial,\lambda)$ are the same scalar multiple of $(\partial+\alpha\lambda+\beta)\phi(\partial)-(\partial+\bar{\alpha}\lambda+\bar{\beta})\phi(\partial+\lambda)$ and $\gamma\phi(\partial)-\bar{\gamma}\phi(\partial+\lambda)$, respectively, where $\phi$ is a polynomial.

\end{enumerate}
\end{lemma}
\begin{proof}
(1) Assume that (\ref{W3}) is a trivial extension, i.e., there exists $v_{\alpha}'=\varphi(\partial)v_{\alpha}+\phi(\partial)v_{\bar{\alpha}}\in E$, where $\varphi(\partial), \phi(\partial) \in \mathbb{C}[\partial]$ and $\varphi(\partial) \neq 0$, such that
\begin{eqnarray*}
L_\lambda v_{\alpha}'=(\partial+\alpha\lambda+\beta)v_{\alpha}'=(\partial+\alpha\lambda+\beta)(\varphi(\partial)v_{\alpha}+\phi(\partial)v_{\bar{\alpha}}),\quad W_\lambda v_{\alpha}'=0.
\end{eqnarray*}
On the other hand, it follows from (\ref{W31}) that
\begin{eqnarray*}
L_\lambda v_{\alpha}'&=&L_\lambda(\varphi(\partial)v_{\alpha}+\phi(\partial)v_{\bar{\alpha}})\\
                                &=&\varphi(\partial+\lambda)L_\lambda v_{\alpha}+\phi(\partial+\lambda)L_\lambda v_{\bar{\alpha}}\\
                                   &=&\varphi(\partial+\lambda)((\partial+\alpha\lambda+\beta)v_{\alpha}+f(\partial,\lambda)v_{\bar{\alpha}})+\phi(\partial+\lambda)(\partial+\bar{\alpha}\lambda+\bar{\beta})v_{\bar{\alpha}}\\
                                   &=&\varphi(\partial+\lambda)(\partial+\alpha\lambda+\beta)v_{\alpha}+(\varphi(\partial+\lambda)f(\partial,\lambda)+\phi(\partial+\lambda)(\partial+\bar{\alpha}\lambda+\bar{\beta}))v_{\bar{\alpha}},\\
W_\lambda v_{\alpha}'&=&\varphi(\partial+\lambda)g(\partial,\lambda)v_{\bar{\alpha}}.
\end{eqnarray*}
We can obtain that $\varphi(\partial)$ is a nonzero constant and $g(\partial,\lambda)=0$ by comparing both expressions for $L_\lambda v_{\alpha}'$ and $W_\lambda v_{\alpha}'$, respectively. Thus $f(\partial,\lambda)$ is a scalar multiple of $(\partial+\alpha\lambda+\beta)\phi(\partial)-(\partial+\bar{\alpha}\lambda+\bar{\beta})\phi(\partial+\lambda)$.

(2) See Corollary 6.3 in \cite{Yuan-Ling}.
\end{proof}

Now, we can consider nontrivial extensions of finite irreducible $\mathcal{W}(a,b)$-modules of the form (\ref{W3}) when $(a,b)\neq (1,0)$.

Applying both sides of (\ref{f31}) and (\ref{f32}) to $v_{\alpha}$ gives the following equations:
\begin{align}
(\lambda-\mu)f(\partial,\lambda+\mu) &=(\partial+\alpha\mu+\lambda+\beta)f(\partial,\lambda)+(\partial+\bar{\alpha}\lambda+\bar{\beta})f(\partial+\lambda,\mu) \nonumber \\
&\quad-(\partial+\alpha\lambda+\mu+\beta)f(\partial,\mu)-(\partial+\bar{\alpha}\mu+\bar{\beta})f(\partial+\mu,\lambda),\label{f332}\\
((a-1)\lambda-\mu+b)g(\partial,\lambda+\mu) &=(\partial+\bar{\alpha}\lambda+\bar{\beta})g(\partial+\lambda,\mu)-(\partial+\alpha\lambda+\mu+\beta)g(\partial,\mu).\label{f333}
\end{align}
Setting $\lambda=0$ in (\ref{f332}) and (\ref{f333}) gives that
\begin{align}
&(\beta-\bar{\beta})f(\partial,\mu) =(\partial+\alpha\mu+\beta)f(\partial,0)-(\partial+\bar{\alpha}\mu+\bar{\beta})f(\partial+\mu,0),\label{f334}\\
&(\beta-\bar{\beta}+b)g(\partial,\mu) =0.\label{f335}
\end{align}

\textbf{Case 1.} $\beta-\bar{\beta} \neq0$, $\beta-\bar{\beta}+b \neq0$.

By (\ref{f334}) and (\ref{f335}), we obtain that $f(\partial,\mu) =\frac{1}{\beta-\bar{\beta}}((\partial+\alpha\mu+\beta)f(\partial,0)-(\partial+\bar{\alpha}\mu+\bar{\beta})f(\partial+\mu,0))$ and $g(\partial,\mu)=0$. This corresponds to the trivial extension by Lemma \ref{l36}(1).

\textbf{Case 2.} $\beta-\bar{\beta} =0$, $b\neq0$.

By (\ref{f335}), we obtain that $g(\partial,\mu)=0$, and it reduces to the case of Virasoro conformal algebra. Then we obtain the result by Theorem \ref{t3}.

\textbf{Case 3.} $\beta-\bar{\beta} \neq0$, $\beta-\bar{\beta}+b=0$, $a \neq1$.

By (\ref{f334}), we obtain that $f(\partial,\mu) =\frac{1}{\beta-\bar{\beta}}((\partial+\alpha\mu+\beta)f(\partial,0)-(\partial+\bar{\alpha}\mu+\bar{\beta})f(\partial+\mu,0))$. Thus $g(\partial,\lambda)\neq 0$. Otherwise, it corresponds to the trivial extension by Lemma \ref{l36}(1). In fact, we can take a shift to let $f(\partial,\mu)=0$ by Lemma \ref{l36}(1).  If $g(\partial,\lambda)=\sum_{n=0}^{m}\sum_{i=0}^{n}a_{ni}\partial^{n-i}\lambda^i$ is the solution of (\ref{f333}), where $a_{ni} \in \mathbb{C}$ and $m$ is the highest degree of $g(\partial,\lambda)$, then $\sum_{i=0}^{m}a_{mi}\partial^{m-i}\lambda^i$ is the solution of the following homogeneous equation:
\begin{align}
((a-1)\lambda-\mu)g(\partial,\lambda+\mu) &=(\partial+\bar{\alpha}\lambda)g(\partial+\lambda,\mu)-(\partial+\alpha\lambda+\mu)g(\partial,\mu).\label{f336}
\end{align}
By Lemma 3.6 in \cite{Ling-Yuan2}, we obtain all solutions of (\ref{f336}) as follows.
\begin{proposition}\label{pro37}(Ref. \cite{Ling-Yuan2}, Lemma 3.6)
Let $g(\partial,\lambda)$ be a nonzero homogeneous polynomial of degree m satisfying (\ref{f336}) with $a\neq1$. Then $\alpha-\bar{\alpha}=m+1-a$ and $m \leq 3$. Furthermore, we have\\
(1) For $a=\frac{5}{3}$, all solutions (up to a scalar) to (\ref{f336}) are given by
\begin{enumerate}[(i)]
\item $m=0, \alpha-\bar{\alpha}=-\frac{2}{3}$, and $g(\partial,\lambda)=1$;
\item $m=1, \alpha-\bar{\alpha}=\frac{1}{3}$, and $g(\partial,\lambda)=\partial+\frac{3}{2}\bar{\alpha}\lambda$;
\item $m=2, \alpha=1, \bar{\alpha}=-\frac{1}{3}$, and $g(\partial,\lambda)=\partial^2+\frac{1}{2}\partial\lambda-\frac{1}{2}\lambda^2$;
\item $m=3, \alpha=\frac{5}{3}, \bar{\alpha}=-\frac{2}{3}$, and $g(\partial,\lambda)=\partial^3+\frac{3}{2}\partial^2\lambda-\frac{3}{2}\partial\lambda^2-\lambda^3$,
\end{enumerate}
(2) For $a\neq\frac{5}{3}$, all solutions (up to a scalar) to (\ref{f336}) are given by
\begin{enumerate}[(i)]
\item $m=0, \alpha-\bar{\alpha}=1-a$, and $g(\partial,\lambda)=1$;
\item $m=1, \alpha-\bar{\alpha}=2-a$, and $g(\partial,\lambda)=\partial-\frac{1}{1-a}\bar{\alpha}\lambda$;
\item $m=2, \alpha=1, \bar{\alpha}=a-2$, and $g(\partial,\lambda)=\partial^2-\frac{1}{1-a}(1+2\bar{\alpha})\partial\lambda-\frac{1}{1-a}\bar{\alpha}\lambda^2$.
\end{enumerate}
\end{proposition}

Therefore, in Case 3, by Proposition \ref{pro37}, to solve (\ref{f333}), we only need to consider the following subcases.

\textbf{Subcase 3.1.} $m=0$.

By Proposition \ref{pro37} and (\ref{f333}),  we can obtain that $\alpha-\bar{\alpha}=1-a$ and $g(\partial,\lambda)=1$.

\textbf{Subcase 3.2.} $m=1$.

By Proposition \ref{pro37}, we can obtain that $\alpha-\bar{\alpha}=2-a$. Assume that $g(\partial,\lambda)=\partial-\frac{1}{1-a}\bar{\alpha}\lambda+a_{00}$. Plugging this into (\ref{f333}) and using undetermined coefficient method, we can obtain that $g(\partial,\lambda)=\partial-\frac{1}{1-a}\bar{\alpha}\lambda+\frac{1}{1-a}\bar{\alpha}b+\bar{\beta}$.

\textbf{Subcase 3.3.} $m=2$.

By Proposition \ref{pro37}, we can obtain that $\alpha=1, \bar{\alpha}=a-2$. Assume that $g(\partial,\lambda)=\partial^2-\frac{1}{1-a}(1+2\bar{\alpha})\partial\lambda-\frac{1}{1-a}\bar{\alpha}\lambda^2+a_{10}\partial+a_{11}\lambda+a_{00}$. Plugging this into (\ref{f333}) and using undetermined coefficient method, we can obtain that $a_{10}=2\bar{\beta}+\frac{1}{1-a}(1+2\bar{\alpha})b$, $a_{11}=\frac{2b}{1-a}\bar{\alpha}-\frac{1}{1-a}(1+2\bar{\alpha})\bar{\beta}$, and $a_{00}=\bar{\beta}^2+b\bar{\beta}\frac{1}{1-a}(1+2\bar{\alpha})-b^2\frac{1}{1-a}\bar{\alpha}$.

\textbf{Subcase 3.4.}  $m=3$.

By Proposition \ref{pro37}, we can obtain that $\alpha=a=\frac{5}{3}, \bar{\alpha}=-\frac{2}{3}$. Assume that $g(\partial,\lambda)=\partial^3+\frac{3}{2}\partial^2\lambda-\frac{3}{2}\partial\lambda^2-\lambda^3+ a_{20}\partial^2+ a_{21}\partial\lambda+a_{22}\lambda^2+a_{10}\partial+a_{11}\lambda+a_{00}$. Plugging this into (\ref{f333}) and using undetermined coefficient method, we can obtain that $a_{20}=3\bar{\beta}-\frac{3}{2}b$, $a_{21}=3\bar{\beta}+3b$, $a_{22}=-\frac{3}{2}\bar{\beta}+3b$, $a_{10}=3\bar{\beta}^2-3b\bar{\beta}-\frac{3}{2}b^2$, $a_{11}=\frac{3}{2}\bar{\beta}^2+3b\bar{\beta}-3b^2$, $a_{00}=\bar{\beta}^3-\frac{3}{2}b\bar{\beta}^2-\frac{3}{2}b^2\bar{\beta}+b^3$.

\textbf{Case 4.} $\beta-\bar{\beta} \neq0$, $\beta-\bar{\beta}+b=0$, $a=1$.

Similar to Case 3, we still have $f(\partial,\mu) =\frac{1}{\beta-\bar{\beta}}((\partial+\alpha\mu+\beta)f(\partial,0)-(\partial+\bar{\alpha}\mu+\bar{\beta})f(\partial+\mu,0))$. Thus $g(\partial,\lambda)\neq 0$. Otherwise, it corresponds to the trivial extension by Lemma \ref{l36}(1). In fact, we can take a shift to let $f(\partial,\mu)=0$ by Lemma \ref{l36}(1). Plugging $a=1$ into (\ref{f333}) gives
\begin{align}
(-\mu+b)g(\partial,\lambda+\mu) &=(\partial+\bar{\alpha}\lambda+\bar{\beta})g(\partial+\lambda,\mu)-(\partial+\alpha\lambda+\mu+\beta)g(\partial,\mu).\label{f337}
\end{align}

If $g(\partial,\lambda)=\sum_{n=0}^{m}\sum_{i=0}^{n}a_{ni}\partial^{n-i}\lambda^i$ is the solution of (\ref{f337}), where $a_{ni} \in \mathbb{C}$ and $m$ is the highest degree of $g(\partial,\lambda)$, then $\sum_{i=0}^{m}a_{mi}\partial^{m-i}\lambda^i$ is the solution of the following homogeneous equation:
\begin{align}
-\mu g(\partial,\lambda+\mu) &=(\partial+\bar{\alpha}\lambda)g(\partial+\lambda,\mu)-(\partial+\alpha\lambda+\mu)g(\partial,\mu).\label{f338}
\end{align}

Setting $\mu=0$ in (\ref{f338}) gives
\begin{align}
(\partial+\bar{\alpha}\lambda)g(\partial+\lambda,0)-(\partial+\alpha\lambda)g(\partial,0)=0.\label{f339}
\end{align}
Obviously, we can obtain that $g(\partial,0)=k$ with $k \in\mathbb{C}$. Therefore, we can obtain that $a_{m0}=0$. If $m\geq 3$, dividing $\mu$ and comparing the coefficients of $\partial^{m-1}\lambda$, $\partial^{m-2}\lambda^2$, $\partial^{m-2}\lambda\mu$, $\partial\lambda^2\mu^{m-3}$, $\lambda\mu^{m-1}$,
$\lambda^i\mu^{m-i}$ $(i\geq 2)$ in (\ref{f338}), respectively, we obtain that
\begin{align}
&a_{m1}=0 \quad or \quad a_{m1}\neq 0 \ and \  \alpha-\bar{\alpha}=m, \label{mf331}\\
&a_{m1}[\bar{\alpha}(m-1)+\frac{(m-1)(m-2)}{2}]=-a_{m2},\label{mf332}\\
&a_{m2}(\bar{\alpha}-\alpha+m)=0,\label{mf333}\\
&(1+2\bar{\alpha})a_{m,m-2}=-\frac{(m-1)(m-2)}{2}a_{m,m-1},\label{mf334}\\
&a_{mm}(\bar{\alpha}-\alpha+m)=0,\label{mf335}\\
&-a_{mm}\tbinom{m}{i}=a_{m,m-i+1}\bar{\alpha}, i=2,3,...,m.\label{mf336}
\end{align}
By (\ref{mf336}), if $a_{m1}=0$, then $a_{mi}=0$, $i=1,2,3,...,m$, i.e., $\sum_{i=0}^{m}a_{mi}\partial^{m-i}\lambda^i=0$, a contradiction. Thus, $a_{m1}\neq0$ and $a_{mi}\neq0$, $i=2,3,...,m$ by (\ref{mf336}). By (\ref{mf332}), (\ref{mf334}) and (\ref{mf336}), we obtain that $m^2-m+2=-2\bar{\alpha}(m-1)$ and $(1+2\bar{\alpha})\tbinom{m}{3}=-\tbinom{m}{2}\tbinom{m-1}{2}$. It follows that $m=3$ and $\bar{\alpha}=-2$. Thus, $m \le 3$.

\textbf{Subcase 4.1.} $m=0$.

Assume that $g(\partial,\lambda)=a_{00}$ with $a_{00}$ is nonzero complex number. Plugging this into (\ref{f337}) and using undetermined coefficient method, we can obtain that $\alpha-\bar{\alpha}=0$ and $g(\partial,\lambda)=1$.

\textbf{Subcase 4.2.} $m=1$.

Assume that $g(\partial,\lambda)=a_{11}\lambda+a_{00}$. Plugging this into (\ref{f337}) and using undetermined coefficient method, we can obtain that $\alpha-\bar{\alpha}=1$ and $g(\partial,\lambda)=\lambda-b$.

\textbf{Subcase 4.3.} $m=2$.

Assume that $g(\partial,\lambda)=a_{21}\partial\lambda+a_{22}\lambda^2+a_{10}\partial+a_{11}\lambda+a_{00}$. Plugging this into (\ref{f337}) and using undetermined coefficient method, we can obtain that $\alpha-\bar{\alpha}=2$ and $g(\partial,\lambda)=\partial\lambda-\bar{\alpha}\lambda^2-b\partial+(\bar{\beta}+2b\bar{\alpha})\lambda-(b\bar{\beta}+b^2\bar{\alpha})$.

\textbf{Subcase 4.4.} $m=3$.

According the above discussion, we obtain that $a_{30}=0$, $\alpha=1$ and $\bar{\alpha}=-2$. Assume that $g(\partial,\lambda)=a_{31}\partial^2\lambda+a_{32}\partial\lambda^2+a_{33}\lambda^3+a_{20}\partial^2+a_{21}\partial\lambda+a_{22}\lambda^2+a_{10}\partial+a_{11}\lambda+a_{00}$. Plugging this into (\ref{f337}) and using undetermined coefficient method, we can obtain that $g(\partial,\lambda)=\partial^2\lambda+3\partial\lambda^2+2\lambda^3-b\partial^2+(2\bar{\beta}-6b)\partial\lambda+(3\bar{\beta}-6b)\lambda^2+(-2\bar{\beta}b+3b^2)\partial+(\bar{\beta}^2-6b\bar{\beta}+6b^2)\lambda-\bar{\beta}^2b+3b^2\bar{\beta}-2b^3$.


\textbf{Case 5.} $\beta-\bar{\beta} =0$, $b=0$.

It reduces to the case of $\mathcal{W}(a,0)$ conformal algebra (Similar to the case of $\mathcal{W}(1-a)$ conformal algebra in \cite{Ling-Yuan2} Theorem 3.7). By Theorem \ref{t3} and Proposition \ref{pro37}, we obtain the following.
 \begin{theorem}\label{t38}(Ref. \cite{Ling-Yuan2}, Theorem 3.7)
 Nontrivial extensions of finite irreducible $\mathcal{W}(a,0)$-modules of the form (\ref{W3}) with $a\neq1$ exist if and only if $\beta=\bar{\beta}$. For each $\beta \in \mathbb{C}$, these extensions are given (up to equivalence) by \ref{W31}, where $g(\partial,\lambda)=0$ and $f(\partial,\lambda)$ is from the nonzero polynomials of Theorem \ref{t3}, with $\alpha, \bar{\alpha}\neq0$, or the values of $\alpha$ and $\bar{\alpha}$ along with the pairs of polynomials $g(\partial,\lambda)$ and $f(\partial,\lambda)$, whose nonzero scalar multiples give rise to nontrivial extensions, are listed as follows (by replacing $\partial$ by $\partial+\beta$):\\
 (1) When $a=3$, we have $\alpha=\bar{\alpha}=1$, $f(\partial,\lambda)=a_0+a_1\lambda$ and $g(\partial,\lambda)=\partial^2+\frac{3}{2}\partial\lambda+\frac{1}{2}\lambda^2$, where $a_0, a_1 \in \mathbb{C}$.\\
 (2) When $a=2$, we have $\alpha-\bar{\alpha}=-1$ or $0$. Moreover,
 \begin{enumerate}[(i)]
 \item In the case $\alpha-\bar{\alpha}=-1$, $f(\partial,\lambda)=0$ and $g(\partial,\lambda)=1$.
 \item In the case $\alpha-\bar{\alpha}=0$, $f(\partial,\lambda)=a_0+a_1\lambda$ and $g(\partial,\lambda)=\partial+\bar{\alpha}\lambda$, where $a_0, a_1 \in \mathbb{C}$.
 \end{enumerate}
(3) When $a=0$, we have $\alpha-\bar{\alpha}=1, 2$ or $\alpha=1, \bar{\alpha}=-2$. Moreover,
 \begin{enumerate}[(i)]
 \item In the case $\alpha-\bar{\alpha}=1$, $f(\partial,\lambda)=0$ and $g(\partial,\lambda)=1$.
 \item In the case $\alpha-\bar{\alpha}=2$, $f(\partial,\lambda)=a_0\lambda^2(2\partial+\lambda)$ and $g(\partial,\lambda)=\partial-\bar{\alpha}\lambda$, where $a_0 \in \mathbb{C}$.
 \item In the case $\alpha=1, \bar{\alpha}=-2$, $f(\partial,\lambda)=a_0\partial\lambda^2(\partial+\lambda)$ and $g(\partial,\lambda)=\partial^2+3\partial\lambda+2\lambda^2$, where $a_0 \in \mathbb{C}$.
 \end{enumerate}
(4) When $a=-1$, we have $\alpha-\bar{\alpha}=2, 3$ or $\alpha=1, \bar{\alpha}=-3$. Moreover,
 \begin{enumerate}[(i)]
 \item In the case $\alpha-\bar{\alpha}=2$, $f(\partial,\lambda)=a_0\lambda^2(2\partial+\lambda)$ and $g(\partial,\lambda)=1$, where $a_0 \in \mathbb{C}$.
 \item In the case $\alpha-\bar{\alpha}=3$, $f(\partial,\lambda)=a_0\partial\lambda^2(\partial+\lambda)$ and $g(\partial,\lambda)=\partial-\frac{1}{2}\bar{\alpha}\lambda$, where $a_0 \in \mathbb{C}$.
 \item In the case $\alpha=1, \bar{\alpha}=-3$, $f(\partial,\lambda)=a_0\lambda^2(4\partial^3+6\partial^2\lambda-\partial\lambda^2-3\lambda^3)$ and $g(\partial,\lambda)=\partial^2+\frac{5}{2}\partial\lambda+\frac{3}{2}\lambda^2$, where $a_0 \in \mathbb{C}$.
 \end{enumerate}
 (5) When $a=-2$, we have $\alpha-\bar{\alpha}=3, 4$ or $\alpha=1, \bar{\alpha}=-4$. Moreover,
 \begin{enumerate}[(i)]
 \item In the case $\alpha-\bar{\alpha}=3$, $f(\partial,\lambda)=a_0\partial\lambda^2(\partial+\lambda)$ and $g(\partial,\lambda)=1$, where $a_0 \in \mathbb{C}$.
 \item In the case $\alpha-\bar{\alpha}=4$, $f(\partial,\lambda)=a_0\lambda^2(4\partial^3+6\partial^2\lambda-\partial\lambda^2+\bar{\alpha}\lambda^3)$ and $g(\partial,\lambda)=\partial-\frac{1}{3}\bar{\alpha}\lambda$, where $a_0 \in \mathbb{C}$.
 \item In the case $\alpha=1, \bar{\alpha}=-4$, $f(\partial,\lambda)=a_0(\partial^4\lambda^2-10\partial^2\lambda^4-17\partial\lambda^5-8\lambda^6)$ and $g(\partial,\lambda)=\partial^2+\frac{7}{3}\partial\lambda+\frac{4}{3}\lambda^2$, where $a_0 \in \mathbb{C}$.
 \end{enumerate}
  (6) When $a=-3$, we have $\alpha-\bar{\alpha}=4, 5$ or $\alpha=1, \bar{\alpha}=-5$. Moreover,
 \begin{enumerate}[(i)]
 \item In the case $\alpha-\bar{\alpha}=4$, $f(\partial,\lambda)=a_0\lambda^2(4\partial^3+6\partial^2\lambda-\partial\lambda^2+\bar{\alpha}\lambda^3)$ and $g(\partial,\lambda)=1$, where $a_0 \in \mathbb{C}$.
 \item In the case $\alpha-\bar{\alpha}=5, \alpha\neq1$, $f(\partial,\lambda)=0$ and $g(\partial,\lambda)=\partial-\frac{1}{4}\bar{\alpha}\lambda$.
 \item In the case $\alpha=1, \bar{\alpha}=-4$, $f(\partial,\lambda)=a_0(\partial^4\lambda^2-10\partial^2\lambda^4-17\partial\lambda^5-8\lambda^6)$ and $g(\partial,\lambda)=\partial+\lambda$, where $a_0 \in \mathbb{C}$.
 \item In the case $\alpha=1, \bar{\alpha}=-5$, $f(\partial,\lambda)=0$ and $g(\partial,\lambda)=\partial^2+\frac{9}{4}\partial\lambda+\frac{5}{4}\lambda^2$.
 \end{enumerate}
 (7) When $a=-4$, we have $\alpha-\bar{\alpha}=5, 6$ or $\alpha=1, \bar{\alpha}=-6$. Moreover,
 \begin{enumerate}[(i)]
 \item In the case $\alpha-\bar{\alpha}=5, \alpha\neq1$, $f(\partial,\lambda)=0$ and $g(\partial,\lambda)=1$.
 \item In the case $\alpha=1, \bar{\alpha}=-4$, $f(\partial,\lambda)=a_0(\partial^4\lambda^2-10\partial^2\lambda^4-17\partial\lambda^5-8\lambda^6)$ and $g(\partial,\lambda)=1$, where $a_0 \in \mathbb{C}$.
 \item In the case $\alpha-\bar{\alpha}=6, \alpha \neq \frac{7}{2}\pm\frac{\sqrt{19}}{2}$, $f(\partial,\lambda)=0$ and $g(\partial,\lambda)=\partial-\frac{1}{5}\bar{\alpha}\lambda$.
  \item In the case $\alpha-\bar{\alpha}=6, \alpha = \frac{7}{2}\pm\frac{\sqrt{19}}{2}$, $f(\partial,\lambda)=a_0(\partial^4\lambda^3-(2\bar{\alpha}+3)\partial^3\lambda^4-3\bar{\alpha}\partial^2\lambda^5-(3\bar{\alpha}+1)\partial\lambda^6-(\bar{\alpha}+\frac{9}{28})\lambda^7)$ and $g(\partial,\lambda)=\partial-\frac{1}{5}\bar{\alpha}\lambda$, where $a_0 \in \mathbb{C}$.
 \item In the case $\alpha=1, \bar{\alpha}=-6$, $f(\partial,\lambda)=0$ and $g(\partial,\lambda)=\partial^2+\frac{11}{5}\partial\lambda+\frac{6}{5}\lambda^2$.
 \end{enumerate}
 (8) When $a=-5$, we have $\alpha-\bar{\alpha}=6, 7$ or $\alpha=1, \bar{\alpha}=-7$. Moreover,
 \begin{enumerate}[(i)]
  \item In the case $\alpha-\bar{\alpha}=6, \alpha \neq \frac{7}{2}\pm\frac{\sqrt{19}}{2}$, $f(\partial,\lambda)=0$ and $g(\partial,\lambda)=1$.
  \item In the case $\alpha-\bar{\alpha}=6, \alpha = \frac{7}{2}\pm\frac{\sqrt{19}}{2}$, $f(\partial,\lambda)=a_0(\partial^4\lambda^3-(2\bar{\alpha}+3)\partial^3\lambda^4-3\bar{\alpha}\partial^2\lambda^5-(3\bar{\alpha}+1)\partial\lambda^6-(\bar{\alpha}+\frac{9}{28})\lambda^7)$ and $g(\partial,\lambda)=1$, where $a_0 \in \mathbb{C}$.
  \item In the case $\alpha-\bar{\alpha}=7$, $f(\partial,\lambda)=0$ and $g(\partial,\lambda)=\partial-\frac{1}{6}\bar{\alpha}\lambda$.
   \item In the case $\alpha=1, \bar{\alpha}=-7$, $f(\partial,\lambda)=0$ and $g(\partial,\lambda)=\partial^2+\frac{13}{6}\partial\lambda+\frac{7}{6}\lambda^2$.
 \end{enumerate}
(9) When $a=\frac{5}{3}$, we have $f(\partial,\lambda)=0$ and the values $\alpha$ and $\bar{\alpha}$ along with $g(\partial,\lambda)$ are from Proposition \ref{pro37}(1).\\
(10) When $a\neq 3, 2, 0, -1, -2, -3, -4, -5$ or $\frac{5}{3}$, we have $f(\partial,\lambda)=0$ and the values $\alpha$ and $\bar{\alpha}$ along with $g(\partial,\lambda)$ are from Proposition \ref{pro37}(2).

 \end{theorem}

Then nontrivial extensions of finite irreducible $\mathcal{W}(1,0)$-modules of the form (\ref{W3}) were classified by Yuan and Ling in Corollary 6.3 in \cite{Yuan-Ling}.

After the above discussion, we can draw the following theorem.
 \begin{theorem}\label{t39}
(A) If $(a,b)\neq (1,0)$, nontrivial extensions of finite irreducible $\mathcal{W}(a,b)$-modules of the form (\ref{W3}) exist. Moreover, they are given (up to equivalence) by (\ref{W31}). The values of $\alpha$ and $\bar{\alpha}$, $\beta$ and $\bar{\beta}$ along with the pairs of polynomials $g(\partial,\lambda)$ and $f(\partial,\lambda)$, whose nonzero scalar multiples give rise to nontrivial extensions, are listed as follows (by replacing $\partial$ by $\partial+\beta$ only in (1) and (4)):

(1) If $\beta-\bar{\beta} =0$, $b\neq0$, then $g(\partial,\lambda)=0$, $f(\partial,\lambda)$ is from the nonzero polynomials of Theorem \ref{t3} with $\alpha, \bar{\alpha} \neq 0$.

(2) If $\beta-\bar{\beta} \neq0$, $\beta-\bar{\beta}+b=0$, $a \neq1$, then $f(\partial,\lambda)=0$ and $g(\partial,\lambda)$ is as follows (where m is the highest degree of $g(\partial,\lambda)$):
\begin{enumerate}[(i)]
\item If $m=0$, then $\alpha-\bar{\alpha}=1-a$ and $g(\partial,\lambda)=1$.
\item If $m=1$, then $\alpha-\bar{\alpha}=2-a$ and $g(\partial,\lambda)=\partial-\frac{1}{1-a}\bar{\alpha}\lambda+\frac{1}{1-a}\bar{\alpha}b+\bar{\beta}$.
\item If $m=2$, then $\alpha=1, \bar{\alpha}=a-2$ and $g(\partial,\lambda)=\partial^2-\frac{1}{1-a}(1+2\bar{\alpha})\partial\lambda-\frac{1}{1-a}\bar{\alpha}\lambda^2+a_{10}\partial+a_{11}\lambda+a_{00}$, where $a_{10}=2\bar{\beta}+\frac{1}{1-a}(1+2\bar{\alpha})b$, $a_{11}=\frac{2b}{1-a}\bar{\alpha}-\frac{1}{1-a}(1+2\bar{\alpha})\bar{\beta}$, and $a_{00}=\bar{\beta}^2+b\bar{\beta}\frac{1}{1-a}(1+2\bar{\alpha})-b^2\frac{1}{1-a}\bar{\alpha}$.
\item If $m=3$, then $\alpha=a=\frac{5}{3}, \bar{\alpha}=-\frac{2}{3}$ and $g(\partial,\lambda)=\partial^3+\frac{3}{2}\partial^2\lambda-\frac{3}{2}\partial\lambda^2-\lambda^3+ a_{20}\partial^2+ a_{21}\partial\lambda+a_{22}\lambda^2+a_{10}\partial+a_{11}\lambda+a_{00}$, where $a_{20}=3\bar{\beta}-\frac{3}{2}b$, $a_{21}=3\bar{\beta}+3b$, $a_{22}=-\frac{3}{2}\bar{\beta}+3b$, $a_{10}=3\bar{\beta}^2-3b\bar{\beta}-\frac{3}{2}b^2$, $a_{11}=\frac{3}{2}\bar{\beta}^2+3b\bar{\beta}-3b^2$, $a_{00}=\bar{\beta}^3-\frac{3}{2}b\bar{\beta}^2-\frac{3}{2}b^2\bar{\beta}+b^3$.
\end{enumerate}

(3) If $\beta-\bar{\beta} \neq0$, $\beta-\bar{\beta}+b=0$, $a=1$, then $f(\partial,\lambda)=0$ and $g(\partial,\lambda)$ is as follows (where m is the highest degree of $g(\partial,\lambda)$):
\begin{enumerate}[(i)]
\item If $m=0$, then $\alpha-\bar{\alpha}=0$ and $g(\lambda)=1$.
\item If $m=1$, then $\alpha-\bar{\alpha}=1$ and $g(\lambda)=\lambda-b$.
\item If $m=2$, then $\alpha-\bar{\alpha}=2$ and $g(\partial,\lambda)=\partial\lambda-\bar{\alpha}\lambda^2-b\partial+(\bar{\beta}+2b\bar{\alpha})\lambda-(b\bar{\beta}+b^2\bar{\alpha})$.
\item If $m=3$, then $\alpha=1$, $\bar{\alpha}=-2$ and $g(\partial,\lambda)=\partial^2\lambda+3\partial\lambda^2+2\lambda^3-b\partial^2+(2\bar{\beta}-6b)\partial\lambda+(3\bar{\beta}-6b)\lambda^2+(-2\bar{\beta}b+3b^2)\partial+(\bar{\beta}^2-6b\bar{\beta}+6b^2)\lambda-\bar{\beta}^2b+3b^2\bar{\beta}-2b^3$.
\end{enumerate}

(4) If $\beta-\bar{\beta} =0$, $b=0$, then $f(\partial,\lambda)$ and $g(\partial,\lambda)$ satisfy the conclusions given in Theorem \ref{t38}.\\ \\
(B) If $(a,b)=(1,0)$, nontrivial extensions of finite irreducible $\mathcal{W}(1,0)$-modules of the form (\ref{W3}) exist if and only if $\gamma=\bar{\gamma}$, $\beta=\bar{\beta}$. Moreover, they are given (up to equivalence) by (\ref{W32}). The values of $\alpha$ and $\bar{\alpha}$, $\beta$ and $\bar{\beta}$, $\gamma$ and $\bar{\gamma}$ along with the pairs of polynomials $g(\partial,\lambda)$ and $f(\partial,\lambda)$, whose nonzero scalar multiples give rise to nontrivial extensions, are listed as follows (by replacing $\partial$ by $\partial+\beta$):\\
(1) If $\gamma =\bar{\gamma}=0$, then $f(\partial,\lambda)$ and $g(\partial,\lambda)$ are as follows:
\begin{enumerate}[(i)]
\item If $\alpha-\bar{\alpha}=0$, then $f(\partial,\lambda)=a_0+a_1\lambda$ and $g(\partial,\lambda)=b_0$ with $(a_0, a_1, b_0) \neq (0,0,0)$.
\item If $\alpha-\bar{\alpha}=1$, then $f(\partial,\lambda)=0$ and $g(\partial,\lambda)=b_1\lambda$ with $b_1\neq 0$.
\item If $\alpha-\bar{\alpha}=2$, then $f(\partial,\lambda)=a_3\lambda^2(2\partial+\lambda)$ and $g(\partial,\lambda)=b_2\lambda(\partial-\bar{\alpha}\lambda)$ with $(a_3, b_2) \neq (0,0)$.

\item If $(\alpha,\bar{\alpha})=(1,-2)$, then $f(\partial,\lambda)=a_4\partial\lambda^2(\partial+\lambda)$ and $g(\partial,\lambda)=b_3\lambda(\partial^2+3\partial\lambda+2\lambda^2)$ with $(a_4, b_3) \neq (0,0)$.
\item If $\alpha-\bar{\alpha}=3$ and $\bar{\alpha} \neq -2$, then $f(\partial,\lambda)=a_4\partial\lambda^2(\partial+\lambda)$ and $g(\partial,\lambda)=0$ with $a_4 \neq 0$.
\item If $\alpha-\bar{\alpha}=4$, then $f(\partial,\lambda)=a_5\lambda^2(4\partial^3+6\partial^2\lambda-\partial\lambda^2+\bar{\alpha}\lambda^3)$ and $g(\partial,\lambda)=0$ with $a_5 \neq 0$.

\item If $(\alpha,\bar{\alpha})=(1,-4)$, then $f(\partial,\lambda)=a_6(\partial^4\lambda^2-10\partial^2\lambda^4-17\partial\lambda^5-8\lambda^6)$ and $g(\partial,\lambda)=0$ with $a_6 \neq 0$.

 \item If $\alpha-\bar{\alpha}=6, \alpha = \frac{7}{2}\pm\frac{\sqrt{19}}{2}$, then $f(\partial,\lambda)=a_7(\partial^4\lambda^3-(2\bar{\alpha}+3)\partial^3\lambda^4-3\bar{\alpha}\partial^2\lambda^5-(3\bar{\alpha}+1)\partial\lambda^6-(\bar{\alpha}+\frac{9}{28})\lambda^7)$ and $g(\partial,\lambda)=0$ with $a_7 \neq 0$.

\end{enumerate}
(2) If $\gamma =\bar{\gamma} \neq0$, then $f(\partial,\lambda)$ and $g(\partial,\lambda)$ are as follows:
 \begin{enumerate}[(i)]

 \item If $\alpha=\bar{\alpha}$, then $f(\partial,\lambda)=a_0+a_1\lambda$ and $g(\partial,\lambda)=b_0$ with $(a_0, a_1, b_0)\neq (0,0,0)$.

 \item If $\alpha-\bar{\alpha}=1$, then $f(\partial,\lambda)=a_2\lambda^2$ and $g(\partial,\lambda)=b_1\lambda$ with $(a_2,b_1)\neq (0,0)$.

 \item If $\alpha-\bar{\alpha}=2$, then $f(\partial,\lambda)=\frac{b_2}{\beta}\partial\lambda^2+a_3\lambda^3$ and $g(\partial,\lambda)=b_2\lambda^2$ with $(b_2,a_3)\neq (0,0)$.
 \end{enumerate}
\end{theorem}

\section{Extensions of finite irreducible modules over $TSV(a,b)$ and $TSV(c)$}

In this section, we apply the methods and results in Section 3 to finite irreducible modules over Lie conformal algebras $TSV(a,b)$ and $TSV(c)$ and give all extensions of finite irreducible modules over them.

\begin{definition}\label{deff1} (Ref. \cite{Hong})
The Lie conformal algebra  $TSV(a,b)$ with two parameters $a,b \in\mathbb{C}$ is a free $\mathbb{C}[\partial]$-module generated by $L$, $Y$ and $M$ and satisfies
 \begin{align*}
 [L_\lambda L]&=(\partial+2\lambda)L,\qquad [L_\lambda Y]=(\partial+a\lambda+b)Y,\\
[L_\lambda M]=(\partial+&2(a-1)\lambda+2b)M,\qquad [Y_\lambda Y]=(\partial+2\lambda)M,\\
[Y_\lambda M]&=[M_\lambda M]=0.
 \end{align*}
 The Lie conformal algebra  $TSV(c)$ with a parameter $c \in\mathbb{C}$ is a free $\mathbb{C}[\partial]$-module generated by $L$, $Y$ and $M$ and satisfies
 \begin{align*}
  [L_\lambda L]&=(\partial+2\lambda)L,\qquad [L_\lambda Y]=(\partial+\frac{3}{2}\lambda+c)Y,\\
[L_\lambda M]=&(\partial+2c)M,\qquad [Y_\lambda Y]=(\partial+2\lambda)(-\partial-2c)M,\\
[Y_\lambda M]&=[M_\lambda M]=0.
 \end{align*}
 \end{definition}
Note that $\mathbb{C}[\partial]M$ is an abelian ideal of both Lie conformal algebras $TSV(a,b)$ and $TSV(c)$. Obviously, we have $TSV(a,b)/\mathbb{C}[\partial]M\cong \mathcal{W}(a,b)$ and $TSV(c)/\mathbb{C}[\partial]M\cong \mathcal{W}(\frac{3}{2},c)$. All finite nontrivial conformal modules over the Lie conformal algebra $TSV(a,b)$ and $TSV(c)$ were classified in \cite{Luo-Hong-Wu}, and the corresponding results are given by the following theorem.

\begin{theorem}\label{t42}(Ref. \cite{Luo-Hong-Wu}, Theorem 4.11)
(1) Any finite nontrivial irreducible $TSV(a,b)$-module $M$ is free of rank one over $\mathbb{C}[\partial]$. Moreover,
\begin{enumerate}[(i)]
\item If $(a,b)\neq (1,0)$,
  \begin{equation*}
  \xymatrix{M\cong M_{\alpha,\beta}=\mathbb{C}[\partial]v,\  L_\lambda v=(\partial+\alpha\lambda+\beta)v,\  Y_\lambda v=M_\lambda v=0,}
   \end{equation*}
with $\alpha,\beta \in \mathbb{C}$  and  $\alpha \neq 0$.

\item If $(a,b)=(1,0)$,
 \begin{equation*}
  \xymatrix{M\cong M_{\alpha,\beta,\gamma}=\mathbb{C}[\partial]v,\  L_\lambda v=(\partial+\alpha\lambda+\beta)v,\  Y_\lambda v=\gamma v,\  M_\lambda v=0,}
 \end{equation*}
with $\alpha,\beta,\gamma \in \mathbb{C}$ and $(\alpha,\gamma) \neq (0,0)$.
\end{enumerate}
(2) Any finite nontrivial irreducible $TSV(c)$-module $M$ is free of rank one over $\mathbb{C}[\partial]$. Moreover,
  \begin{equation*}
  \xymatrix{M\cong M_{\alpha,\beta}=\mathbb{C}[\partial]v,\  L_\lambda v=(\partial+\alpha\lambda+\beta)v,\  Y_\lambda v=M_\lambda v=0,}
   \end{equation*}
with $\alpha,\beta \in \mathbb{C}$  and  $\alpha \neq 0$.
\end{theorem}

Denote the module $M$ in Theorem \ref{t42}(1) by $M_{\alpha,\beta}$ (respectively, $M_{\alpha,\beta,\gamma}$) if $(a,b)\neq (1,0)$ (respectively, $(a,b)=(1,0)$). Denote the module $M$ in Theorem \ref{t42}(2) by $M_{\alpha,\beta}$.

By Definition \ref{Def22}, a $TSV(a,b)$-module structure on $M$ is given by $L_{\lambda}, Y_{\lambda}, M_{\lambda} \in End_{\mathbb{C}}(M)[\lambda]$ such that
\begin{align}
&[L_{\lambda},L_{\mu}]=(\lambda-\mu)L_{\lambda+\mu}, \label{f41}\\
&[L_{\lambda},Y_{\mu}]=((a-1)\lambda-\mu+b)Y_{\lambda+\mu}, \label{f42}\\
&[L_{\lambda},M_{\mu}]=((2a-3)\lambda-\mu+2b)M_{\lambda+\mu}, \label{f43}\\
&[Y_{\lambda},Y_{\mu}]=(\lambda-\mu)M_{\lambda+\mu}, \label{f44}\\
&[\partial,L_{\lambda}]=-\lambda L_{\lambda}, \label{f45}\\
&[\partial,Y_{\lambda}]=-\lambda Y_{\lambda}, \label{f46}\\
&[\partial,M_{\lambda}]=-\lambda M_{\lambda}, \label{f47}\\
&[Y_{\lambda},M_{\mu}]=0, \label{f48}\\
&[M_{\lambda},M_{\mu}]=0. \label{f49}
\end{align}
A $TSV(c)$-module structure on $M$ is given by $L_{\lambda}, Y_{\lambda}, M_{\lambda} \in End_{\mathbb{C}}(M)[\lambda]$ such that
\begin{align}
&[L_{\lambda},L_{\mu}]=(\lambda-\mu)L_{\lambda+\mu}, \label{f51}\\
&[L_{\lambda},Y_{\mu}]=(\frac{1}{2}\lambda-\mu+c)Y_{\lambda+\mu}, \label{f52}\\
&[L_{\lambda},M_{\mu}]=(-\lambda-\mu+2c)M_{\lambda+\mu}, \label{f53}\\
&[Y_{\lambda},Y_{\mu}]=(\lambda-\mu)(\lambda+\mu-2c)M_{\lambda+\mu}, \label{f54}\\
&[\partial,L_{\lambda}]=-\lambda L_{\lambda}, \label{f55}\\
&[\partial,Y_{\lambda}]=-\lambda Y_{\lambda}, \label{f56}\\
&[\partial,M_{\lambda}]=-\lambda M_{\lambda}, \label{f57}\\
&[Y_{\lambda},M_{\mu}]=0, \label{f58}\\
&[M_{\lambda},M_{\mu}]=0. \label{f59}
\end{align}

First, we consider extensions of finite irreducible modules over $TSV(a,b)$ and $TSV(c)$ of the form
\begin{align}
0\longrightarrow \mathbb{C}c_{\eta} \longrightarrow  E\longrightarrow  M \longrightarrow   0\label{T1}
\end{align}
As before, $E$ as a $\mathbb{C}[\partial]$-module in (\ref{T1}) is isomorphic to $\mathbb{C}c_{\eta}\oplus M$, where $\mathbb{C}c_{\eta}$ is a $TSV(a,b)$ (resp. $TSV(c)$)-submodule, and $M=\mathbb{C}[\partial]v_{\alpha}$ such that the following identities hold in $E$:\\
(1) In the $TSV(a,b)$ case,
\begin{enumerate}[(i)]
\item If $(a,b)\neq (1,0)$,
  \begin{equation}
  L_\lambda v_{\alpha}=(\partial+\alpha\lambda+\beta)v_{\alpha}+f(\lambda)c_{\eta},\quad Y_\lambda v_{\alpha}=g(\lambda)c_{\eta},\quad M_\lambda v_{\alpha}=h(\lambda)c_{\eta};\label{TAB11}
   \end{equation}
  \item If $(a,b)=(1,0)$,
 \begin{equation}
  L_\lambda v_{\alpha}=(\partial+\alpha\lambda+\beta)v_{\alpha}+f(\lambda)c_{\eta},\quad Y_\lambda v_{\alpha}=\gamma v_{\alpha}+g(\lambda)c_{\eta},\quad M_\lambda v_{\alpha}=h(\lambda)c_{\eta};\label{TAB12}
  \end{equation}
\end{enumerate}
where $f(\lambda), g(\lambda), h(\lambda) \in \mathbb{C}[\lambda]$.\\
(2) In the $TSV(c)$ case,
  \begin{equation}
  L_\lambda v_{\alpha}=(\partial+\alpha\lambda+\beta)v_{\alpha}+f(\lambda)c_{\eta},\quad Y_\lambda v_{\alpha}=g(\lambda)c_{\eta},\quad M_\lambda v_{\alpha}=h(\lambda)c_{\eta},\label{TC11}
   \end{equation}
where $f(\lambda), g(\lambda), h(\lambda) \in \mathbb{C}[\lambda]$.

\begin{lemma}\label{l43}
(1) All trivial extensions of finite irreducible $TSV(a,b)$-modules of the form (\ref{T1}) are given by (\ref{TAB11}) and (\ref{TAB12}), and
\begin{enumerate}[(i)]
\item If $(a,b)\neq (1,0)$, $f(\lambda)$ is a scalar multiple of $\alpha\lambda+\beta+\eta$ and $g(\lambda)=h(\lambda)=0$.
\item If $(a,b)=(1,0)$, $f(\lambda)$ and $g(\lambda)$ are the same scalar multiple of $\alpha\lambda+\beta+\eta$ and $\gamma$, respectively, and $h(\lambda)=0$.
\end{enumerate}
(2) All trivial extensions of finite irreducible $TSV(c)$-modules of the form (\ref{T1}) are given by (\ref{TC11}), where $f(\lambda)$ is a scalar multiple of $\alpha\lambda+\beta+\eta$ and $g(\lambda)=h(\lambda)=0$.
\end{lemma}
\begin{proof}
Similar to the proof of Lemma \ref{l32}.
\end{proof}

\begin{theorem}\label{t44}
(1) In the $TSV(a,b)$ case,

\ (i) If $(a,b)\neq (1,0)$, nontrivial extensions of finite irreducible $TSV(a,b)$-modules of the form (\ref{T1}) exist if and only if $h(\lambda)=0$. Moreover, they are given (up to equivalence) by (\ref{TAB11}). The values of $\beta$ and $\eta$ along with the pairs of polynomials $g(\lambda)$ and $f(\lambda)$, whose nonzero scalar multiples give rise to nontrivial extensions, are listed as follows:

\begin{enumerate}[(a)]
\item if $g(\lambda)=0$, then $\alpha=1,2$, $\beta+\eta=0$ and $f(\lambda)$ is from the nonzero polynomials of Theorem \ref{t25};
\item if $a\neq1$, $b=0$ and $\beta+\eta =0$, then $g(\lambda)=k$ for some nonzero complex number $k$, $\alpha=1-a$, and

$$f(\lambda)=
\begin{cases}
c_2{\lambda}^2,  &\alpha=1,
\cr c_3\lambda^3,  &\alpha=2,
\cr 0,  &otherwise,
\end{cases}$$
with $c_2, c_3 \in \mathbb{C}$;

\item if $a\neq1$, $b+\beta+\eta=0$ and $\beta+\eta \neq0$, then $g(\lambda)=k$ for some nonzero complex number $k$, $\alpha=1-a$, and $f(\lambda)=0$;
\item if $a=1, b\neq0$ and $b+\beta+\eta=0$, then $g(\lambda)=k(1-\frac{1}{b}\lambda)$ for some nonzero complex number $k$, $\alpha=1$, and $f(\lambda)=0$.
\end{enumerate}

(ii) If $(a,b)=(1,0)$, nontrivial extensions of finite irreducible $TSV(1,0)$-modules of the form (\ref{T1}) exist if and only if $\beta+\eta=0$, $\gamma=0$ and $h(\lambda)=0$. Moreover, they are given (up to equivalence) by (\ref{TAB12}), where, if $g(\lambda)=0$, then $\alpha=1,2$ and $f(\lambda)$ is from the nonzero polynomials of Theorem \ref{t25}, or else $g(\lambda)=k\lambda$ for some nonzero complex number $k$, $\alpha=1$ and $f(\lambda)=c_2{\lambda}^2$ with $c_2 \in \mathbb{C}$.

(2) In the $TSV(c)$ case, nontrivial extensions of finite irreducible $TSV(c)$-modules of the form (\ref{T1}) exist if and only if $h(\lambda)=0$. Moreover, they are given (up to equivalence) by (\ref{TC11}). The values of $\beta$ and $\eta$ along with the pairs of polynomials $g(\lambda)$ and $f(\lambda)$, whose nonzero scalar multiples give rise to nontrivial extensions, are listed as follows:

\begin{enumerate}[(a)]
\item if $g(\lambda)=0$, then $\alpha=1,2$, $\beta+\eta=0$ and $f(\lambda)$ is from the nonzero polynomials of Theorem \ref{t25};
\item if $c+\beta+\eta=0$, then $g(\lambda)=k$ for some nonzero complex number $k$, $\alpha=-\frac{1}{2}$ and $f(\lambda)=0$.

\end{enumerate}

\end{theorem}
\begin{proof}
(1) (i) Applying both sides of (\ref{f44}) to $v_{\alpha}$, we obtain $(\lambda-\mu)h(\lambda+\mu)=0$. Thus, $h(\lambda)=0$. It reduces to the case of $\mathcal{W}(a,b)$. We obtain the result by Theorem \ref{t33}(1).

(ii) Applying both sides of (\ref{f43}), (\ref{f44}) and (\ref{f48}) to $v_{\alpha}$ gives
\begin{align}
&(\alpha\lambda+\mu+\beta+\eta)h(\mu)=(\lambda+\mu)h(\lambda+\mu), \label{f423} \\
&\gamma(g(\lambda)-g(\mu))=(\lambda-\mu)h(\lambda+\mu), \label{f424} \\
&\gamma h(\mu)=0. \label{f425}
\end{align}
Obviously, we can deduce that $h(\lambda)=0$. Otherwise, we obtain that $h(\lambda)=k$ by (\ref{f423}), where $k$ is a nonzero complex number. By (\ref{f424}) and (\ref{f425}), we obtain a contradiction. Thus, $h(\lambda)=0$ and it reduces to the case of $\mathcal{W}(a,b)$. We obtain the result by Theorem \ref{t33}(2).\\
(2) Similar to the proof of (1)(i), we can deduce that $h(\lambda)=0$. It reduces to the case of $\mathcal{W}(\frac{3}{2},c)$. We obtain the result by Theorem \ref{t33}(1).

This completes the proof.
\end{proof}
Next, we consider extensions of finite irreducible modules over $TSV(a,b)$ and $TSV(c)$ of the form
\begin{align}
0\longrightarrow M \longrightarrow  E\longrightarrow  \mathbb{C}c_{\eta} \longrightarrow   0\label{T2}
\end{align}
As before, $E$ as a vector space in (\ref{T2}) is isomorphic to $ M\oplus \mathbb{C}c_{\eta}$, where $M$ is a $TSV(a,b)$(resp. $TSV(c)$)-submodule, and $M=\mathbb{C}[\partial]v_{\alpha}$ such that the following identities hold in $E$:
\begin{equation}
  L_\lambda c_{\eta}=f(\partial,\lambda)v_{\alpha},\quad Y_\lambda c_{\eta}=g(\partial,\lambda)v_{\alpha},\quad M_\lambda c_{\eta}=h(\partial,\lambda)v_{\alpha},\quad \partial c_{\eta}=\eta c_{\eta} +p(\partial)v_{\alpha}, \label{TSV2}
 \end{equation}
 where $f(\partial,\lambda), g(\partial,\lambda), h(\partial,\lambda) \in \mathbb{C}[\partial,\lambda]$ and $p(\partial) \in \mathbb{C}[\partial]$.

\begin{lemma}\label{l45}
(1) All trivial extensions of finite irreducible $TSV(a,b)$-modules of the form (\ref{T2}) are given by (\ref{TSV2}), and
\begin{enumerate}[(i)]
\item If $(a,b)\neq (1,0)$, $f(\partial,\lambda)=(\partial+\alpha\lambda+\beta)\phi(\partial+\lambda)$, $g(\partial,\lambda)=h(\partial,\lambda)=0$ and $p(\partial)=(\partial-\eta)\phi(\partial)$, where $\phi$ is a polynomial.
\item If $(a,b)=(1,0)$, $f(\partial,\lambda)=(\partial+\alpha\lambda+\beta)\phi(\partial+\lambda)$, $g(\partial,\lambda)=\gamma \phi(\partial+\lambda)$, $h(\partial,\lambda)=0$ and $p(\partial)=(\partial-\eta)\phi(\partial)$, where $\phi$ is a polynomial.
\end{enumerate}
(2) All trivial extensions of finite irreducible $TSV(c)$-modules of the form (\ref{T2}) are given by (\ref{TSV2}), where $f(\partial,\lambda)=(\partial+\alpha\lambda+\beta)\phi(\partial+\lambda)$, $g(\partial,\lambda)=h(\partial,\lambda)=0$ and $p(\partial)=(\partial-\eta)\phi(\partial)$, where $\phi$ is a polynomial.
\end{lemma}
\begin{proof}
Similar to the proof of Lemma \ref{l34}.
\end{proof}

\begin{theorem}\label{t46}
(1) In the $TSV(a,b)$ case,

\ (i) If $(a,b)\neq (1,0)$, nontrivial extensions of finite irreducible $TSV(a,b)$-modules of the form (\ref{T2}) exist if and only if $\beta+\eta=0$ and $\alpha=1$. In this case, $dim Ext(\mathbb{C}c_{-\beta}, M_{1,\beta})=1$, and the unique (up to equivalence) nontrivial extension is given by
\begin{align*}
L_\lambda c_{\eta}=kv_{\alpha},\quad Y_\lambda c_{\eta}=M_\lambda c_{\eta}=0,\quad \partial c_{\eta}=\eta c_{\eta}+kv_{\alpha},
\end{align*}
where $k$ is a nonzero complex number.

\ (ii) If $(a,b)=(1,0)$, nontrivial extensions of finite irreducible $TSV(1,0)$-modules of the form (\ref{T2}) exist if and only if $\beta+\eta=0$ and $(\alpha,\gamma)=(1,0)$. In this case, $dim Ext(\mathbb{C}c_{-\beta}, M_{1,\beta,0})=1$, and the unique (up to equivalence) nontrivial extension is given by
\begin{align*}
L_\lambda c_{\eta}=kv_{\alpha},\quad Y_\lambda c_{\eta}=M_\lambda c_{\eta}=0,\quad \partial c_{\eta}=\eta c_{\eta}+kv_{\alpha},
\end{align*}
where $k$ is a nonzero complex number.\\
(2) In the $TSV(c)$ case, nontrivial extensions of finite irreducible $TSV(c)$-modules of the form (\ref{T2}) exist if and only if $\beta+\eta=0$ and $\alpha=1$. In this case, $dim Ext(\mathbb{C}c_{-\beta}, M_{1,\beta})=1$, and the unique (up to equivalence) nontrivial extension is given by
\begin{align*}
L_\lambda c_{\eta}=kv_{\alpha},\quad Y_\lambda c_{\eta}=M_\lambda c_{\eta}=0,\quad \partial c_{\eta}=\eta c_{\eta}+kv_{\alpha},
\end{align*}
where $k$ is a nonzero complex number.
\end{theorem}
\begin{proof}
Applying both sides of (\ref{f47}) to $c_{\eta}$ gives the following equations:
\begin{eqnarray}
&&(\partial+\lambda-\eta)h(\partial,\lambda)=0. \label{f428}
\end{eqnarray}
Obviously, $h(\partial,\lambda)=0$ by (\ref{f428}). Thus, all cases reduce to the case of $\mathcal{W}(a,b)$(resp. $\mathcal{W}(\frac{3}{2},c)$). We obtain the result by Theorem \ref{t35}.

 This completes the proof.
\end{proof}


Finally, we consider extensions of finite irreducible modules over $TSV(a,b)$ and $TSV(c)$ of the form
\begin{align}
0\longrightarrow \bar{M} \longrightarrow  E\longrightarrow  M \longrightarrow   0\label{T3}
\end{align}
As before, $E$ as a $\mathbb{C}[\partial]$-module in (\ref{T3}) is isomorphic to $ \bar{M}\oplus M$, where $\bar{M}$ is a $TSV(a,b)$(resp. $TSV(c)$)-submodule, and $\bar{M}=\mathbb{C}[\partial]v_{\bar{\alpha}}$,  $M=\mathbb{C}[\partial]v_{\alpha}$ such that the following identities hold in $E$:\\
(1) In the $TSV(a,b)$ case,
\begin{enumerate}[(i)]
\item If $(a,b)\neq (1,0)$,
  \begin{equation}
  L_\lambda v_{\alpha}=(\partial+\alpha\lambda+\beta)v_{\alpha}+f(\partial,\lambda)v_{\bar{\alpha}},\quad Y_\lambda v_{\alpha}=g(\partial,\lambda)v_{\bar{\alpha}},\quad M_\lambda v_{\alpha}=h(\partial,\lambda)v_{\bar{\alpha}};\label{TAB31}
   \end{equation}

  \item If $(a,b)=(1,0)$,
 \begin{equation}
 L_\lambda v_{\alpha}=(\partial+\alpha\lambda+\beta)v_{\alpha}+f(\partial,\lambda)v_{\bar{\alpha}},\quad Y_\lambda v_{\alpha}=\gamma v_{\alpha}+g(\partial,\lambda)v_{\bar{\alpha}},\quad M_\lambda v_{\alpha}=h(\partial,\lambda)v_{\bar{\alpha}};\label{TAB32}
  \end{equation}
\end{enumerate}
where $f(\partial,\lambda), g(\partial,\lambda), h(\partial,\lambda) \in \mathbb{C}[\partial,\lambda]$.\\
(2) In the $TSV(c)$ case,
 \begin{equation}
 L_\lambda v_{\alpha}=(\partial+\alpha\lambda+\beta)v_{\alpha}+f(\partial,\lambda)v_{\bar{\alpha}},\quad Y_\lambda v_{\alpha}=\gamma v_{\alpha}+g(\partial,\lambda)v_{\bar{\alpha}},\quad M_\lambda v_{\alpha}=h(\partial,\lambda)v_{\bar{\alpha}};\label{TC31}
  \end{equation}
where $f(\partial,\lambda), g(\partial,\lambda), h(\partial,\lambda) \in \mathbb{C}[\partial,\lambda]$.

\begin{lemma}\label{l47}
(1) All trivial extensions of finite irreducible $TSV(a,b)$-modules of the form (\ref{T3}) are given by (\ref{TAB31}) and (\ref{TAB32}), and
\begin{enumerate}[(i)]
\item If $(a,b)\neq (1,0)$, $f(\partial,\lambda)$ is a scalar multiple of $(\partial+\alpha\lambda+\beta)\phi(\partial)-(\partial+\bar{\alpha}\lambda+\bar{\beta})\phi(\partial+\lambda)$ and $g(\partial,\lambda)=h(\partial,\lambda)=0$, where $\phi$ is a polynomial.
\item If $(a,b)=(1,0)$, $f(\partial,\lambda)$ and $g(\partial,\lambda)$ are the same scalar multiple of $(\partial+\alpha\lambda+\beta)\phi(\partial)-(\partial+\bar{\alpha}\lambda+\bar{\beta})\phi(\partial+\lambda)$ and $\gamma\phi(\partial)-\bar{\gamma}\phi(\partial+\lambda)$, respectively, where $\phi$ is a polynomial, and $h(\partial,\lambda)=0$.
\end{enumerate}

(2) All trivial extensions of finite irreducible $TSV(c)$-modules of the form (\ref{T3}) are given by (\ref{TC31}), where $f(\partial,\lambda)$ is a scalar multiple of $(\partial+\alpha\lambda+\beta)\phi(\partial)-(\partial+\bar{\alpha}\lambda+\bar{\beta})\phi(\partial+\lambda)$ and $g(\partial,\lambda)=h(\partial,\lambda)=0$, where $\phi$ is a polynomial.
\end{lemma}
\begin{proof}
Similar to the proof of Lemma \ref{l36}.
\end{proof}

\begin{theorem}\label{t48}
(1) In the $TSV(a,b)$ case,\\
(A) If $(a,b)\neq (1,0)$, nontrivial extensions of finite irreducible $TSV(a,b)$-modules of the form (\ref{T3}) exist if and only if $h(\partial,\lambda)=0$. Moreover, they are given (up to equivalence) by (\ref{TAB31}). The values of $\alpha$ and $\bar{\alpha}$, $\beta$ and $\bar{\beta}$ along with the pairs of polynomials $g(\partial,\lambda)$ and $f(\partial,\lambda)$, whose nonzero scalar multiples give rise to nontrivial extensions, are listed as follows (by replacing $\partial$ by $\partial+\beta$ only in (i) and (iv)):\\
(i) If $\beta-\bar{\beta} =0$, $b\neq0$, then $g(\partial,\lambda)=0$, $f(\partial,\lambda)$ is from the nonzero polynomials of Theorem \ref{t3} with $\alpha, \bar{\alpha} \neq 0$.\\
(ii) If $\beta-\bar{\beta} \neq0$, $\beta-\bar{\beta}+b=0$, $a \neq1$, then $f(\partial,\lambda)=0$ and $g(\partial,\lambda)$ is as follows (where $m$ is the highest degree of $g(\partial,\lambda)$):

\begin{enumerate}[(a)]
\item If $m=0$, then $\alpha-\bar{\alpha}=1-a$ and $g(\partial,\lambda)=1$.
\item If $m=1$, then $\alpha-\bar{\alpha}=2-a$ and $g(\partial,\lambda)=\partial-\frac{1}{1-a}\bar{\alpha}\lambda+\frac{1}{1-a}\bar{\alpha}b+\bar{\beta}$.
\item If $m=2$, then $\alpha=1, \bar{\alpha}=a-2$ and $g(\partial,\lambda)=\partial^2-\frac{1}{1-a}(1+2\bar{\alpha})\partial\lambda-\frac{1}{1-a}\bar{\alpha}\lambda^2+a_{10}\partial+a_{11}\lambda+a_{00}$, where $a_{10}=2\bar{\beta}+\frac{1}{1-a}(1+2\bar{\alpha})b$, $a_{11}=\frac{2b}{1-a}\bar{\alpha}-\frac{1}{1-a}(1+2\bar{\alpha})\bar{\beta}$, and $a_{00}=\bar{\beta}^2+b\bar{\beta}\frac{1}{1-a}(1+2\bar{\alpha})-b^2\frac{1}{1-a}\bar{\alpha}$.
\item If $m=3$, then $\alpha=a=\frac{5}{3}, \bar{\alpha}=-\frac{2}{3}$ and $g(\partial,\lambda)=\partial^3+\frac{3}{2}\partial^2\lambda-\frac{3}{2}\partial\lambda^2-\lambda^3+ a_{20}\partial^2+ a_{21}\partial\lambda+a_{22}\lambda^2+a_{10}\partial+a_{11}\lambda+a_{00}$, where $a_{20}=3\bar{\beta}-\frac{3}{2}b$, $a_{21}=3\bar{\beta}+3b$, $a_{22}=-\frac{3}{2}\bar{\beta}+3b$, $a_{10}=3\bar{\beta}^2-3b\bar{\beta}-\frac{3}{2}b^2$, $a_{11}=\frac{3}{2}\bar{\beta}^2+3b\bar{\beta}-3b^2$, $a_{00}=\bar{\beta}^3-\frac{3}{2}b\bar{\beta}^2-\frac{3}{2}b^2\bar{\beta}+b^3$.
\end{enumerate}
(iii) If $\beta-\bar{\beta} \neq0$, $\beta-\bar{\beta}+b=0$, $a=1$, then $f(\partial,\lambda)=0$ and $g(\partial,\lambda)$ is as follows (where $m$ is the highest degree of $g(\partial,\lambda)$):

\begin{enumerate}[(a)]
\item If $m=0$, then $\alpha-\bar{\alpha}=0$ and $g(\lambda)=1$.
\item If $m=1$, then $\alpha-\bar{\alpha}=1$ and $g(\lambda)=\lambda-b$.
\item If $m=2$, then $\alpha-\bar{\alpha}=2$ and $g(\partial,\lambda)=\partial\lambda-\bar{\alpha}\lambda^2-b\partial+(\bar{\beta}+2b\bar{\alpha})\lambda-(b\bar{\beta}+b^2\bar{\alpha})$.
\item If $m=3$, then $\alpha=1$, $\bar{\alpha}=-2$ and $g(\partial,\lambda)=\partial^2\lambda+3\partial\lambda^2+2\lambda^3-b\partial^2+(2\bar{\beta}-6b)\partial\lambda+(3\bar{\beta}-6b)\lambda^2+(-2\bar{\beta}b+3b^2)\partial+(\bar{\beta}^2-6b\bar{\beta}+6b^2)\lambda-\bar{\beta}^2b+3b^2\bar{\beta}-2b^3$.
\end{enumerate}
(iv) If $\beta-\bar{\beta} =0$, $b=0$, then $f(\partial,\lambda)$ and $g(\partial,\lambda)$ satisfy the conclusions given in Theorem \ref{t38}.\\

(B) If $(a,b)=(1,0)$, nontrivial extensions of finite irreducible $TSV(a,b)$-modules of the form (\ref{T3}) exist if and only if $\gamma=\bar{\gamma}$, $\beta=\bar{\beta}$. Moreover, they are given (up to equivalence) by (\ref{TAB32}). The values of $\alpha$ and $\bar{\alpha}$, $\beta$ and $\bar{\beta}$, $\gamma$ and $\bar{\gamma}$ along with the pairs of polynomials $h(\partial,\lambda)$, $g(\partial,\lambda)$ and $f(\partial,\lambda)$, whose nonzero scalar multiples give rise to nontrivial extensions, are listed as follows (by replacing $\partial$ by $\partial+\beta$):\\
(i) If $\gamma =\bar{\gamma}=0$, then $h(\partial,\lambda)=0$, $f(\partial,\lambda)$ and $g(\partial,\lambda)$ are as follows:
\begin{enumerate}[(a)]
\item If $\alpha-\bar{\alpha}=0$, then $f(\partial,\lambda)=a_0+a_1\lambda$ and $g(\partial,\lambda)=b_0$ with $(a_0, a_1, b_0) \neq (0,0,0)$.
\item If $\alpha-\bar{\alpha}=1$, then $f(\partial,\lambda)=0$ and $g(\partial,\lambda)=b_1\lambda$ with $b_1\neq 0$.
\item If $\alpha-\bar{\alpha}=2$, then $f(\partial,\lambda)=a_3\lambda^2(2\partial+\lambda)$ and $g(\partial,\lambda)=b_2\lambda(\partial-\bar{\alpha}\lambda)$ with $(a_3, b_2) \neq (0,0)$.

\item If $(\alpha,\bar{\alpha})=(1,-2)$, then $f(\partial,\lambda)=a_4\partial\lambda^2(\partial+\lambda)$ and $g(\partial,\lambda)=b_3\lambda(\partial^2+3\partial\lambda+2\lambda^2)$ with $(a_4, b_3) \neq (0,0)$.
\item If $\alpha-\bar{\alpha}=3$ and $\bar{\alpha} \neq -2$, then $f(\partial,\lambda)=a_4\partial\lambda^2(\partial+\lambda)$ and $g(\partial,\lambda)=0$ with $a_4 \neq 0$.
\item If $\alpha-\bar{\alpha}=4$, then $f(\partial,\lambda)=a_5\lambda^2(4\partial^3+6\partial^2\lambda-\partial\lambda^2+\bar{\alpha}\lambda^3)$ and $g(\partial,\lambda)=0$ with $a_5 \neq 0$.

\item If $(\alpha,\bar{\alpha})=(1,-4)$, then $f(\partial,\lambda)=a_6(\partial^4\lambda^2-10\partial^2\lambda^4-17\partial\lambda^5-8\lambda^6)$ and $g(\partial,\lambda)=0$ with $a_6 \neq 0$.

 \item If $\alpha-\bar{\alpha}=6, \alpha = \frac{7}{2}\pm\frac{\sqrt{19}}{2}$, then $f(\partial,\lambda)=a_7(\partial^4\lambda^3-(2\bar{\alpha}+3)\partial^3\lambda^4-3\bar{\alpha}\partial^2\lambda^5-(3\bar{\alpha}+1)\partial\lambda^6-(\bar{\alpha}+\frac{9}{28})\lambda^7)$ and $g(\partial,\lambda)=0$ with $a_7 \neq 0$.

\end{enumerate}
(ii) If $\gamma =\bar{\gamma} \neq0$ and $h(\partial,\lambda)=0$, then $f(\partial,\lambda)$ and $g(\partial,\lambda)$ are as follows:
 \begin{enumerate}[(a)]

 \item If $\alpha=\bar{\alpha}$, then $f(\partial,\lambda)=a_0+a_1\lambda$ and $g(\partial,\lambda)=b_0$ with $(a_0, a_1, b_0)\neq (0,0,0)$.

 \item If $\alpha-\bar{\alpha}=1$, then $f(\partial,\lambda)=a_2\lambda^2$ and $g(\partial,\lambda)=b_1\lambda$ with $(a_2,b_1)\neq (0,0)$.

 \item If $\alpha-\bar{\alpha}=2$, then $f(\partial,\lambda)=\frac{b_2}{\beta}\partial\lambda^2+a_3\lambda^3$ and $g(\partial,\lambda)=b_2\lambda^2$ with $(b_2,a_3)\neq (0,0)$.
 \end{enumerate}
(iii) If $\gamma =\bar{\gamma} \neq0$ and $h(\partial,\lambda)\neq0$, then $f(\partial,\lambda)$ and $g(\partial,\lambda)$ are as follows:
\begin{enumerate}[(a)]

 \item If $\alpha=1$ and $\bar{\alpha}=0$, then $f(\partial,\lambda)=b_1\lambda^2$, $g(\partial,\lambda)=\frac{k}{\gamma}\partial$, $h(\partial,\lambda)=k$, where $k\neq 0$, $b_1 \in \mathbb{C}$.

 \item If $\alpha-\bar{\alpha}=1$ and $\bar{\alpha}\neq 0$, then $f(\partial,\lambda)=0$, $g(\partial,\lambda)=\frac{k}{\gamma}\partial+c_2\lambda$, $h(\partial,\lambda)=k$, where $k\neq 0$, $c_2 \in \mathbb{C}$. \\
 \end{enumerate}

(2) In the $TSV(c)$ case, nontrivial extensions of finite irreducible $TSV(c)$-modules of the form (\ref{T3}) exist if and only if $h(\partial,\lambda)=0$. Moreover, they are given (up to equivalence) by (\ref{TC31}). The values of $\alpha$ and $\bar{\alpha}$, $\beta$ and $\bar{\beta}$ along with the pairs of polynomials $g(\partial,\lambda)$ and $f(\partial,\lambda)$, whose nonzero scalar multiples give rise to nontrivial extensions, are listed as follows (by replacing $\partial$ by $\partial+\beta$ only in (i) and (iii)):\\
(i) If $\beta-\bar{\beta} =0$, $c\neq0$, then $g(\partial,\lambda)=0$, $f(\partial,\lambda)$ is from the nonzero polynomials of Theorem \ref{t3} with $\alpha, \bar{\alpha} \neq 0$.\\
(ii) If $\beta-\bar{\beta} \neq0$, $\beta-\bar{\beta}+c=0$,  then $f(\partial,\lambda)=0$ and $g(\partial,\lambda)$ is as follows (where $m$ is the highest degree of $g(\partial,\lambda)$):

\begin{enumerate}[(a)]
\item If $m=0$, then $\alpha-\bar{\alpha}=-\frac{1}{2}$ and $g(\partial,\lambda)=1$.
\item If $m=1$, then $\alpha-\bar{\alpha}=\frac{1}{2}$ and $g(\partial,\lambda)=\partial+2\bar{\alpha}\lambda-2\bar{\alpha}c+\bar{\beta}$.
\item If $m=2$, then $\alpha=1, \bar{\alpha}=-\frac{1}{2}$ and $g(\partial,\lambda)=\partial^2+2(1+2\bar{\alpha})\partial\lambda+2\bar{\alpha}\lambda^2+a_{10}\partial+a_{11}\lambda+a_{00}$, where $a_{10}=2\bar{\beta}-2(1+2\bar{\alpha})c$, $a_{11}=-4c\bar{\alpha}+2(1+2\bar{\alpha})\bar{\beta}$, and $a_{00}=\bar{\beta}^2-2c\bar{\beta}(1+2\bar{\alpha})+2c^2\bar{\alpha}$.
\end{enumerate}
(iii) If $\beta-\bar{\beta} =0$, $c=0$, then $f(\partial,\lambda)$ and $g(\partial,\lambda)$ satisfy the conclusions given in Theorem \ref{t38} with $a=\frac{3}{2}$.\\
\end{theorem}
\begin{proof}
(1)(A) Applying both sides of (\ref{f44}) to $v_{\alpha}$ gives
\begin{align*}
(\lambda-\mu)h(\partial,\lambda+\mu)=0
\end{align*}
Obviously, $h(\partial,\lambda)=0$. Then it reduces to the case of $\mathcal{W}(a,b)$. We obtain the result by Theorem \ref{t39}(A).

(B) Applying both sides of (\ref{f41}), (\ref{f42}), (\ref{f43}), (\ref{f44}) and (\ref{f48}) to $v_{\alpha}$ gives the following equations:
\begin{align}
(\lambda-\mu)f(\partial,\lambda+\mu) &=(\partial+\alpha\mu+\lambda+\beta)f(\partial,\lambda)+(\partial+\bar{\alpha}\lambda+\bar{\beta})f(\partial+\lambda,\mu) \nonumber \\
&\quad-(\partial+\alpha\lambda+\mu+\beta)f(\partial,\mu)-(\partial+\bar{\alpha}\mu+\bar{\beta})f(\partial+\mu,\lambda),\label{f433}\\
-\mu g(\partial,\lambda+\mu) &=(\partial+\bar{\alpha}\lambda+\bar{\beta})g(\partial+\lambda,\mu)-(\partial+\alpha\lambda+\mu+\beta)g(\partial,\mu)\nonumber\\
&\quad+\gamma f(\partial,\lambda)- \bar{\gamma} f(\partial+\mu,\lambda),\label{f434}\\
-(\lambda+\mu)h(\partial,\lambda+\mu)&=(\partial+\bar{\alpha}\lambda+\bar{\beta})h(\partial+\lambda,\mu)-(\partial+\alpha\lambda+\mu+\beta)h(\partial,\mu),\label{f435}\\
\gamma(g(\partial,\lambda)-g(\partial,\mu))&+\bar{\gamma} (g(\partial+\lambda,\mu)-g(\partial+\mu,\lambda))=(\lambda-\mu)h(\partial,\lambda+\mu),\label{f436}\\
\bar{\gamma}h(\partial+\lambda,\mu)&-\gamma h(\partial,\mu)=0.\label{f437}
\end{align}

\textbf{Case 1.} $\gamma \neq \bar{\gamma}$.

By (\ref{f437}), we obtain that $h(\partial,\lambda)=0$. It reduces to the case of $\mathcal{W}(1,0)$. It corresponds to the trivial extension by Theorem \ref{t39}(B).

\textbf{Case 2.} $\gamma =\bar{\gamma}=0$.

By (\ref{f436}), we obtain that $h(\partial,\lambda)=0$. It reduces to the case of $\mathcal{W}(1,0)$. We obtain the result by Theorem \ref{t39}(B).

\textbf{Case 3.} $\gamma =\bar{\gamma} \neq0$.

If $h(\partial,\lambda)=0$, it reduces to the case of $\mathcal{W}(1,0)$. We obtain the result by Theorem \ref{t39}(B).

If $h(\partial,\lambda)\neq 0$, we can obtain that $h(\partial,\lambda)=h(\lambda)$ by (\ref{f437}). Plugging this into (\ref{f435}) gives
\begin{align}
-(\lambda+\mu) h(\lambda+\mu) =((\bar{\alpha}-\alpha)\lambda-\mu+\bar{\beta}-\beta) h(\mu).\label{f438}
\end{align}
We obtain that $h(\partial,\lambda)=k$ by (\ref{f438}), where $k$ is a nonzero complex number. Plugging this into (\ref{f435}) again, we obtain that $\alpha-\bar{\alpha}=1$ and $\beta-\bar{\beta}=0$. By Theorem \ref{t3} and (\ref{f433}), we can deduce that $f(\partial,\lambda)=a_0\partial+b_0\partial\lambda+b_1\lambda^2$, where $a_0,b_0,b_1\in \mathbb{C}$ if $\alpha=1$ and $\bar{\alpha}=0$ or $f(\partial,\lambda)=0$. For convenience, we put $\bar{\partial}=\partial+\beta$ and let $\bar{f}(\bar{\partial},\lambda)=f(\bar{\partial}-\beta,\lambda)$, $\bar{g}(\bar{\partial},\lambda)=g(\bar{\partial}-\beta,\lambda)$ and $\bar{h}(\bar{\partial},\lambda)=h(\bar{\partial}-\beta,\lambda)$. In what follows we will continue to write $\partial$ for $\bar{\partial}$, $f$ for $\bar{f}$, $g$ for $\bar{g}$ and $h$ for $\bar{h}$. Now we can rewrite (\ref{f433}), (\ref{f434}) and (\ref{f436}) as follows:
\begin{align}
(\lambda-\mu)f(\partial,\lambda+\mu) &=(\partial+\alpha\mu+\lambda)f(\partial,\lambda)+(\partial+\bar{\alpha}\lambda)f(\partial+\lambda,\mu) \nonumber \\
&\quad-(\partial+\alpha\lambda+\mu)f(\partial,\mu)-(\partial+\bar{\alpha}\mu)f(\partial+\mu,\lambda),\label{f439}\\
-\mu g(\partial,\lambda+\mu) &=(\partial+\bar{\alpha}\lambda)g(\partial+\lambda,\mu)-(\partial+\alpha\lambda+\mu)g(\partial,\mu)\nonumber\\
&\quad+\gamma (f(\partial,\lambda)-f(\partial+\mu,\lambda)),\label{f440}\\
\gamma(g(\partial,\lambda)-g(\partial,\mu)&+g(\partial+\lambda,\mu)-g(\partial+\mu,\lambda))=k(\lambda-\mu).\label{f441}
\end{align}

If $\alpha=1$, $\bar{\alpha}=0$ and $f(\partial,\lambda)=a_0\partial+b_0\partial\lambda+b_1\lambda^2$, where $a_0,b_0,b_1\in \mathbb{C}$, plugging this into (\ref{f440}) gives
\begin{align}
-\mu g(\partial,\lambda+\mu) &=\partial g(\partial+\lambda,\mu)-(\partial+\lambda+\mu)g(\partial,\mu)-\gamma (a_0+b_0\lambda)\mu.\label{f442}
\end{align}
Setting $\lambda=0$ in (\ref{f442}) gives $a_0\gamma  \mu=0$. Thus, $a_0=0$. Setting $m=\text{deg} g(\partial,\lambda)$. If $m > 1$, the homogeneous part of degree $m$ in $g(\partial,\lambda)$ meets the following equation:
\begin{align}
-\mu g(\partial,\lambda+\mu) &=\partial g(\partial+\lambda,\mu)-(\partial+\lambda+\mu)g(\partial,\mu).\label{f443}\\
g(\partial,\lambda)-g(\partial,\mu)&+g(\partial+\lambda,\mu)-g(\partial+\mu,\lambda)=0.\label{f444}
\end{align}
Similar to solve the equation (\ref{f338}), there are no solution of (\ref{f443}) when $m>1$ and $\alpha-\bar{\alpha}=1$. Thus, $m\le1$. Setting $\partial=\mu=0$ in (\ref{f442}) gives $g(0,0)=0$. Assume that $g(\partial,\lambda)=c_1\partial+c_2\lambda$, plugging this into (\ref{f441}) and (\ref{f442}) gives $a_0=b_0=0$, $c_1=\frac{k}{\gamma}\neq 0$ and $b_1, c_2 \in \mathbb{C}$. By Lemma \ref{l47}(1)(ii), $f(\partial,\lambda)=h(\partial,\lambda)=0$ and $g(\partial,\lambda)=c_2\lambda$ corresponds to the trivial extension. Then we can assume that $c_2=0$. Thus, $f(\partial,\lambda)=b_1\lambda^2$, $g(\partial,\lambda)=\frac{k}{\gamma}\partial$, $h(\partial,\lambda)=k$, $\alpha=1$, $\bar{\alpha}=0$, $\beta-\bar{\beta}=0$ and $\gamma=\bar{\gamma}\neq 0$, where $k\neq 0$, $b_1 \in \mathbb{C}$.

If $\alpha-\bar{\alpha}=1$ and $\bar{\alpha}\neq 0$, $f(\partial,\lambda)=0$, $g(\partial,\lambda)=\frac{k}{\gamma}\partial+c_2\lambda$, $h(\partial,\lambda)=k$, $\beta-\bar{\beta}=0$ and $\gamma=\bar{\gamma}\neq 0$, where $k\neq 0$, $c_2 \in \mathbb{C}$. \\
(2) Similar to the proof of (1)(A), we can deduce that $h(\partial,\lambda)=0$. Thus, it reduces to the case of $\mathcal{W}(\frac{3}{2},c)$. We obtain the result by Theorem \ref{t39}(A).

This completes the proof.
\end{proof}

\end{document}